\documentclass[a4paper,12pt]{article}

\usepackage{amsthm,amsmath,amscd,mathtools}
\usepackage[headheight=14.5pt,top=2cm,headsep=.5cm,vmargin=2cm,hmargin=2cm]{geometry}
\usepackage{hyperref}
\usepackage[utf8]{inputenc}
\usepackage[autobold]{mathfixs}
\usepackage{stackrel}
\usepackage{cases}
\usepackage[charter]{mathdesign}
\usepackage{caption}
\usepackage{xcolor}
\usepackage{graphicx}
\DeclareMathSizes{12}{12}{8}{6}

\usepackage{cite}
\usepackage{enumitem}
\setlist[itemize,2]{label=$\centerdot$}
\setlist[itemize,3]{label=$\triangle$}

\newtheoremstyle{ptheorem}{1em}{0em}{\itshape}{}{\bfseries}{.}{.5em}{\thmname{#1}\thmnumber{
		#2}\thmnote{ (\hspace{-.01pt}{#3})}}

\theoremstyle{ptheorem}

\newtheorem{thm}{Theorem}[section]
\newtheorem{pro}[thm]{Proposition}
\newtheorem{lem}[thm]{Lemma}
\newtheorem{cor}[thm]{Corollary}

\newtheoremstyle{hdef}{1em}{0em}{}{}{\bfseries}{.}{.5em}{\thmname{#1}\thmnumber{
		#2}\thmnote{ (\hspace{-.01pt}{#3})}}
\theoremstyle{hdef}

\newtheorem{dfn}[thm]{Definition}
\newtheorem{rem}[thm]{Remark}

\newtheorem{exa}[thm]{Example}

\numberwithin{equation}{section}
\numberwithin{figure}{section}

\DeclareMathOperator{\Id}{Id}

\DeclareMathOperator{\dif}{d}

\newcommand{\cB}{{\mathcal B}}
\newcommand{\cC}{{\mathcal C}}

\newcommand{\bF}{{\mathbb F}}

\newcommand{\bN}{{\mathbb N}}

\newcommand{\bR}{{\mathbb R}}

\newcommand{\e}{\varepsilon}

\renewcommand{\phi}{\varphi}
\renewcommand{\le}{\leqslant}
\renewcommand{\ge}{\geqslant}

\newcommand{\n}{{n\in\bN}}

\renewcommand{\d}{\delta}

\renewcommand{\(}{\left(}
\renewcommand{\)}{\right)}

\newcommand{\til}{\widetilde}

\newcommand{\bs}{\backslash}

\newcommand{\olb}[1]{%
	\vbox{\offinterlineskip\ialign{\hfil##\hfil\cr
			$\rotatebox[origin=c]{90}{$]$}$\cr\noalign{\kern-.45ex}{$#1$}\cr}}}

\newcommand{\noop}[1]{}

\renewcommand{\ss}{\subset}
\usepackage{stmaryrd}

\newcommand{\irchi}{\mbox{\raisebox{1.5pt}{$\chi$}}}

\allowdisplaybreaks

\begin{document}

	\title{Consequences of the product rule in Stieltjes differentiability}
	\author{Francisco J. Fernández$^{*}$, Ignacio Márquez Albés$^{**\dagger}$ and F. Adri\'an F. Tojo$^{*\dagger}$}
	\date{}

	\author{
		F. Javier Fernández$^{*}$\\
		\normalsize e-mail: fjavier.fernandez@usc.es\\
		Ignacio Márquez Albés$^{*\dagger}$ \\
		\normalsize e-mail: ignacio.marquez@usc.es \\
		F. Adri\'an F. Tojo$^{*\dagger}$ \\
		\normalsize e-mail: fernandoadrian.fernandez@usc.es\\ \normalsize \emph{$^{*}$Departamento de Estatística, Análise Matemática e Optimización},\\ \normalsize \emph{Universidade de Santiago de Compostela, 15782, Facultade de Matemáticas, Santiago, Spain.}\\\normalsize
		\\\normalsize \emph{$^{\dagger}$CITMAga, 15782, Santiago de Compostela, Spain}}

\maketitle

\begin{abstract}
	This work revolves around the study of differentiability in the Stieltjes sense of a product of functions. A formula for the first order derivative has been obtained in the past, which is similar to the usual one with some extra terms in its expression. The aim of this paper is to take this behavior into account to study under which conditions we can guarantee the existence of higher order derivatives, while obtaining some other interesting results for the Stieltjes derivative along the way. We also investigate the regularity of the product of two functions.
\end{abstract}



	\section{Introduction}

	The notion of differentiating with respect to a function (Stieltjes differentiation) is a rather classical concept --see \cite{Daniell,young1917,lebesgue1928leccons}-- that has its roots at the genesis of calculus when, in an intuitive way, it was usual to think of a quantity varying continuously with respect to another (dependent or independent). More recently, a new important application of Stieltjes calculus was found: it serves as a powerful bridge between continuous and discrete calculus and, importantly, differential equations. In particular, it was shown that other methods of unifying discrete and differential calculus, such as time scales or equations with impulses, can be thought of as particular instances of Stieltjes calculus \cite{LoRo14}. This theory has the advantage of building on the powerful results of measure theory but, at the same time, allowing for classical solutions and explicit computations. From there, several authors have revitalized the theory by providing a solid theoretical framework \cite{PoMa,FriLo17} and different applications of Stieltjes differential equations such as models of silk worm populations \cite{PoMa}, fishing models with open and closed seasons \cite{LoMa19}, culture and growth of cyanobacteria \cite{LoMa19Resolution,LoMaMon18}, solvent solution and water evaporation models \cite{FP2016}, among others.

	In those previous works there was no definition of the Stieltjes derivative which made sense at every point of the domain. Instead, they provided a definition at almost every point for a certain measure, being thus comparable to the Radon-Nikodym derivative. Recently, in \cite{Fernandez2021}, the definition was extended in order to include all points of the domain, which also had an important repercussion in previously known results such as the product rule, which, in this more general context, reads
	\[	\left(f_{1} f_{2}\right)_{g}'(t)=\left(f_{1}\right)_{g}'(t) f_{2}(t^*)+\left(f_{2}\right)_{g}'(t) f_{1}(t^*)+\left(f_{1}\right)_{g}'(t)\left(f_{2}\right)_{g}'(t) \Delta g(t^*),\]
	where $g$ is the function defining the Stieltjes derivative and $t^*$ is a point that depends on $t$.

	This expression has immediate consequences that make Stieltjes Calculus different from the usual setting. It is evident from the expression that the regularity of the functions $f_1$ and $f_2$ in the expression is not enough to guarantee the same level of regularity to the product as, once we differentiate for a first time, 
	the terms $f_{1}(t^*)$, $f_{2}(t^*)$ and $\Delta g(t^*)$ appear. This fact begs a few questions: \emph{given a map $f$ that is $g$-differentiable at a point $t$, is the map $f^*(t):=f(t^*)$ $g$-differentiable at $t$? Is there a relation between the $g$-derivatives of the two functions?} Similarly, we need to concern ourselves with the study of the $g$-differentiability of the map $\Delta g^*$ and its role in the product rule. It would seem that, at least, $\Delta g$ ought to be always $g$-differentiable as its regularity is intrinsically defined by the same function $g$ with respect to which we are differentiating, but this is far from reality, as previously known examples show --cf. \cite[Remark~3.16]{Fernandez2021}. Hence, we need to ask ourselves, \emph{when is $\Delta g^*$ $g$-differentiable?}


	Having a clear understanding of these issues is vital in order to deepen in the study of the classes of continuously $g$-differentiable functions and, in particular, in order to ascertain the role of the product of functions in these classes for, as it is already known, we will not enjoy the usual algebra structure with the product of functions as algebra operation.

	With these objectives in mind, we set our course to better understand the differentiability of $\Delta g$ and the role of $t^*$ throughout the following sections. In Section 2 we establish same basic definitions and results that will be necessary in order to understand and develop further results. It is with Section 3 that we derive the $g$-differentiability properties of $\Delta g$ and provide counterexamples that show the optimality of our results. 
			In Section 4 we show under which conditions the product of two functions is, at least, two times $g$-differentiable. To that end, we study how evaluating on $t^*$ instead of $t$ affects a given function with regard to its $g$-differentiability and, in particular, the function $\Delta g$. 
 Finally, in Section 5 we study the continuity of the $g$-derivative of the product.
	\section{Preliminaries}
	Let $g:\mathbb R\to\mathbb R$ be a nondecreasing and left-continuous function, which we call \emph{derivator}, and denote by $\mathbb F$ the field $\mathbb R$ or $\mathbb C$.
	We shall write as 
	$\mu_g$ the Lebesgue-Stieltjes measure associated to $g$ given by
	\[\mu_g([c,d))=g(d)-g(c),\quad c,d\in\mathbb R,\ c<d,\]
	see \cite{Ru87,Sche97,Burk07}. It is important to remark that $\mu_g$ is a Borel measure.
	 We will use the term ``$g$-measurable'' for a set or function to refer to $\mu_g$-measurability in the corresponding sense, and we denote by $\mathcal L^1_{g}(X,\mathbb F)$ the set of Lebesgue-Stieltjes $\mu_g$-integrable functions on a $g$-measurable set $X$ with values in $\mathbb F$, whose integral we write as
	\[\int_X f(s)\,\dif\mu_g(s),\quad f\in\mathcal L^1_{g}(X,\mathbb F).\]
	Similarly, we will talk about properties holding \emph{$g$-almost everywhere} in a set $X$ (shortened to \emph{$g$-a.e.} in~$X$), or holding for \emph{$g$-almost all} (or, simply, \emph{$g$-a.a.}) $x\in X$, as a simplified way to express that they hold $\mu_g$-almost everywhere in $X$ or for $\mu_g$-almost all $x\in X$, respectively.

	Define the sets
	\begin{align*}
		C_g&=\{ t \in \mathbb R \, : \, \mbox{$g$ is constant on $(t-\varepsilon,t+\varepsilon)$ for some $\varepsilon>0$} \},\\
		D_g&=\{ t \in \mathbb R \, : \, \Delta g(t)>0\},
	\end{align*}
	where $\Delta g(t):=g(t^+)-g(t)$, $t\in\mathbb R$, and $g(t^+)$ denotes the right handside limit of $g$ at $t$. First, observe that $C_g\cap D_g=\emptyset$. Furthermore, as pointed out in \cite{LoRo14}, the set $C_g$ is open in the usual topology of the real line, so it can be uniquely expressed as a countable union of open disjoint intervals, say
	\begin{equation}\label{Cgdisj}
		C_g=\bigcup_{n\in\Lambda} (a_n,b_n).
	\end{equation}
	where $\Lambda\ss\bN$. With this notation, we introduce
	the sets $N_g^-$ and $N_g^+$ in \cite{LoMa19Resolution}, defined as
	\[N_g^-=\{a_n: n\in\Lambda\}\backslash D_g,\quad N_g^+=\{b_n:n\in\Lambda\}\backslash D_g,\quad N_g=N_g^-\cup N_g^+.\]
	\begin{rem}\label{pointcharact}
		It is important to remark for the work ahead that, by definition, for any $t\in\mathbb R\backslash C_g$, at least one of the following conditions must hold:
		\begin{align}
			g(s)<g(t),\quad\mbox{for all }s\in[a,b],\ s<t,\label{cond1}\\
			g(s)>g(t),\quad\mbox{for all }s\in[a,b],\ s>t.\label{cond2}
		\end{align}
		In particular, for $t\in N_g^-$ only~\eqref{cond1} holds; and, similarly, for $t\in N_g^+$ only~\eqref{cond2} holds. On the other hand, if $t\in D_g$, then~\eqref{cond2} always holds and~\eqref{cond1} fails only when $t=b_n$, $n\in\mathbb N$, for some $b_n$ as in~\eqref{Cgdisj}. For the remaining cases, i.e. when $t\in\mathbb R\backslash (C_g\cup N_g\cup D_g)$, both~\eqref{cond1} and~\eqref{cond2} hold.
	\end{rem}

	Let us recall the following definition of Stieltjes derivative in \cite[Definition 3.7]{Fernandez2021}. 
	To that end, we consider $a,b\in\mathbb R$, $a<b$, such that $a\not\in N_g^-$ and $b\not\in D_g\cup C_g\cup N_g^+$. 
	A careful reader might observe that throughout the entirety of \cite{Fernandez2021} it is also required that $g(a)=0$ and $a\not\in D_g$. The first of these conditions can easily be avoided by redefining the map $g$ if necessary; whereas the condition $a\not\in D_g$ can be imposed without loss of generality, as pointed out by \cite{Fernandez2021,FriLo17} and \cite[Proposition~4.28]{MarquezTesis}, when the focus of the study is the existence and uniqueness of solution of differential problems, which is not our case. Nevertheless, this condition is not required for the following definition.

	\begin{dfn}[{\cite[Definition 3.7]{Fernandez2021}}]
		We define the \emph{Stieltjes derivative}\index{Stieltjes derivative}, or \emph{$g$-derivative}\index{$g$-derivative}, of a map $f:[a,b]\to\mathbb F$ at a point $t\in [a,b]$ as
		\[
		f'_g(t)=\left\{
		\begin{array}{ll}
			\displaystyle \lim_{s \to t}\frac{f(s)-f(t)}{g(s)-g(t)},\quad & t\not\in D_{g}\cup C_g,\vspace{0.1cm}\\
			\displaystyle\lim_{s\to t^+}\frac{f(s)-f(t)}{g(s)-g(t)},\quad & t\in D_{g},\vspace{0.1cm}\\
			\displaystyle\lim_{s\to b_n^+}\frac{f(s)-f(b_n)}{g(s)-g(b_n)},\quad & t\in C_{g},\ t\in(a_n,b_n),
		\end{array}
		\right.
		\]
		where $a_n, b_n$ are as in~\eqref{Cgdisj}, provided the corresponding limits exist. In that case, we say that $f$ is \emph{$g$-differentiable at $t$}. 
	\end{dfn}

	\begin{rem}\label{remNgderivative}
		For $t\in N_g\cup\{a,b\}$,
		the corresponding limit in the definition of $g$-derivative at $t$ must be understood in the sense explained in \cite[Remark 2.2]{Ma21}, that is, the Stieltjes derivative in such points is computed as
		\begin{equation*} 
			f'_g(t)=\left\{\begin{array}{ll}
				\displaystyle \displaystyle \lim_{s \to t^+}\frac{f(s)-f(t)}{g(s)-g(t)}, & t\in N_g^+\cup\{a\}, \vspace{0.1cm} \\
				\displaystyle \displaystyle \lim_{s \to t^-}\frac{f(s)-f(t)}{g(s)-g(t)}, & t\in N_g^-\cup\{b\},
			\end{array}\right.
		\end{equation*}
		provided the corresponding limit exists.
	\end{rem}
	\begin{rem}\label{remDgCg}
		It follows from the definition that, for $t\in D_g$, $f'_g(t)$ exists if and only if $f(t^+)$ exists and, in that case,
		\begin{equation*}
			f'_g(t)=\frac{f(t^+)-f(t)}{\Delta g(t)}.
		\end{equation*}
		Similarly, if $t\in (a_n,b_n)\subset C_g$ for some $a_n, b_n$ in~\eqref{Cgdisj}, we have that $f'_g(t)$ exists if and only if $f'_g(b_n)$ exists and, in that case, $f'_g(t)=f'_g(b_n)$.
	\end{rem}


	It is possible to further simplify the definition of Stieltjes derivative at a point $t\in[a,b]$ by defining 
	\begin{equation}\label{tstar}
		t^*=
		\begin{dcases}
			t,\quad & t\not\in C_g,\vspace{0.1cm}\\
			b_n,\quad & t\in (a_n,b_n)\subset C_g,
		\end{dcases}
	\end{equation}
	with $a_n,b_n$ as in~\eqref{Cgdisj}. With this notation, we have that 
	\[
	f'_g(t)=\left\{
	\begin{array}{ll}
		\displaystyle \lim_{s \to t}\frac{f(s)-f(t)}{g(s)-g(t)},\quad & t\not\in D_{g}\cup C_g,\vspace{0.1cm}\\
		\displaystyle\lim_{s\to t^{*+}}\frac{f(s)-f(t^*)}{g(s)-g(t^*)},\quad & t\in D_{g}\cup C_g,
	\end{array}
	\right.
	\]
	provided the corresponding limit exists. Note that the information in Remark~\ref{remNgderivative} should still be taken into account.

	For completeness, we include the following result contaning some basic properties of $t^*$ that follow directly from the definition and the assumptions on the point $b$.
	\begin{lem}Let $t\in[a,b]$ and let $t^*$ be the corresponding point in~\eqref{tstar}. Then, we have that $t^*\in[t,b]$ and $g(t)=g(t^*)$. Furthermore,
		\begin{itemize}
			\item if $t\in C_g$, then $t^*\not\in C_g$ (and so $t^*\ne t$);
			\item if $t\in [a,b]\bs C_g$, then $t=t^*$;
			\item if $t^*\in [a,b]\bs \{b_n:n\in\Lambda\}$, then $t=t^*$.
		\end{itemize}
		In particular, if we denote by $t^{**}$ the corresponding point in~\eqref{tstar} for $t^*$, it holds that $t^{**}=t^*$.
	\end{lem}


	The following result can be found in \cite[Proposition~3.9]{Fernandez2021} and it includes some basic properties of the Stieltjes derivative. Observe that this is a generalization of \cite[Proposition~3.13]{MarquezTesis}, which is only stated for real-valued functions which are $g$-differentiable at a point of $\mathbb R\backslash C_g$.
	\begin{pro}\label{PropStiDer} 
		Let $t\in[a,b]$.	 If $f_1,f_2:[a,b]\to\mathbb F$ are $g$-differentiable at $t$, then:
		\begin{itemize}
			\item The function $\lambda_{1} f_{1}+\lambda_{2} f_{2}$ is $g$-differentiable at $t$ for any $\lambda_{1}, \lambda_{2} \in \mathbb{R}$ and
			\begin{equation*}%
				\left(\lambda_{1} f_{1}+\lambda_{2} f_{2}\right)_{g}'(t)=\lambda_{1}\left(f_{1}\right)_{g}'(t)+\lambda_{2}\left(f_{2}\right)_{g}'(t).
			\end{equation*}
			\item The product $f_{1} f_{2}$ is $g$-differentiable at $t$ and
			\begin{equation}\label{deltaproducto}
				\left(f_{1} f_{2}\right)_{g}'(t)=\left(f_{1}\right)_{g}'(t) f_{2}(t^*)+\left(f_{2}\right)_{g}'(t) f_{1}(t^*)+\left(f_{1}\right)_{g}'(t)\left(f_{2}\right)_{g}'(t) \Delta g(t^*).
			\end{equation}
			\item If $f_2(t^*)\,(f_2(t^*)+(f_2)'_g(t)\, \Delta g(t^*))\neq 0$, the quotient $f_1/f_2$ is $g$-differentiable at $t$ and 
			\begin{equation*}
				\left(\frac{f_1}{f_2}\right)'_g(t)=\frac{\left(f_{1}\right)_{g}'(t) f_{2}(t^*)-\left(f_{2}\right)_{g}'(t) f_{1}(t^*)}{f_2(t^*)\,(f_2(t^*)+(f_2)'_g(t)\, \Delta g(t^*))}
			\end{equation*}
		\end{itemize}
	\end{pro}

		Naturally, it is possible to extend the product rule for a finite number of functions. To that end, we introduce the following notation: for $n,k\in\bN$ and $(x_1,\dots,x_n)\in \{0,1\}^n$ define 
		\begin{align*}
			\left\lvert (x_1,\dots,x_n)\right\rvert:=\operatorname{card}\{k\in\{1,\dots, n\}\ :\ x_k=1\},\quad F_{n,k}:=\{\sigma\in \{0,1\}^n\ :\ \left\lvert \sigma\right\rvert=k\}.
		\end{align*}
		We also denote $f^{(0)}_g\equiv f$, $f^{(1)}_g\equiv f'_g$.

		\begin{pro}[Product rule] Let $n\in\mathbb N$, $n\ge 2$, and $f_j:[a,b]\to\bF$, $j=1,2,\dots,n$, be $g$-differentiable at $t\in[a,b]$. Then, the product $\prod_{j=1}^nf_j$ is also $g$-differentiable at $t$ and
			\begin{equation}\label{gderivativenproduct}
				\(\prod_{j=1}^nf_j\)_g'(t)=\sum_{k=0}^{n-1}(\Delta g(t^*))^k\(\sum_{\sigma\in F_{n,k+1}}\prod_{j=1}^n(f_j)^{(\sigma_j)}_g(t^*)\).
			\end{equation}
		\end{pro}
		\begin{rem}
			Observe that the term $\Delta g(t^*)^k$ might not be properly defined when $k=0$. This is because~\eqref{gderivativenproduct} is just a concise way of writing
			\begin{equation*}
				\(\prod_{j=1}^nf_j\)_g'(t)=\sum_{\sigma\in F_{n,1}}\prod_{j=1}^n(f_j)^{(\sigma_j)}_g(t^*)+\sum_{k=1}^{n-1}(\Delta g(t^*))^k\(\sum_{\sigma\in F_{n,k+1}}\prod_{j=1}^n(f_j)^{(\sigma_j)}_g(t^*)\).
			\end{equation*}
			In other words, in~\eqref{gderivativenproduct} we implicitly treat $\Delta g(t^*)^0$ as $1$. We adopt this notation for the rest of the section.
		\end{rem}
		\begin{proof}
			We prove the result by induction on $n\in\mathbb N$.
%
			For $n=2$, Proposition~\ref{PropStiDer} ensures that $f_1\cdot f_2$ is $g$-differentiable at $t$ and
			\[\(f_1\cdot f_2\)_g'(t)=(f_1)_g^{(1)}(t) f_2^{(0)}(t^*)+f_1^{(0)}(t^*) (f_2)_g^{(1)}(t)+(f_1)_g^{(1)}(t) (f_2)_g^{(1)}(t)\Delta g(t^*).\]
			Now, by definition of $t^*$, we have that $f^{(1)}_g(t)=f^{(1)}_g(t^*)$ and $f^{(2)}_g(t)=f^{(2)}_g(t^*)$, so it follows that~\eqref{gderivativenproduct} holds.

			Assume that the result is true for $n-1$. Then, Proposition~\ref{PropStiDer} guarantees that $\prod_{j=1}^nf_j$ is $g$-differentiable at $t$ and, since~\eqref{gderivativenproduct} holds for $n=2$, omitting the evaluation at $t^*$, 
			we have that
			\begin{align*}
				\(\prod_{j=1}^{n}f_j\)_g'= & 
				\(\prod_{j=1}^{n-1}f_j\)_g^{(1)}\cdot f_n^{(0)}+\(\prod_{j=1}^{n-1}f_j\)^{(0)}\cdot (f_n)_g^{(1)}+\(\prod_{j=1}^{n-1}f_j\)_g^{(1)}\cdot(f_n)_g^{(1)}\cdot\Delta g
				\\
				= &
				\left(\sum_{k=0}^{n-2}(\Delta g)^k\(\sum_{\sigma\in F_{n-1,k+1}}\prod_{j=1}^{n-1}(f_j)^{(\sigma_j)}_g\)\right)\cdot f_n^{(0)}+\(\prod_{j=1}^{n-1}f_k^{(0)}\)\cdot (f_n)_g^{(1)}
				\\ & 
				+\left[\sum_{k=0}^{n-2}(\Delta g)^k\(\sum_{\sigma\in F_{n-1,k+1}}\prod_{j=1}^{n-1}(f_j)^{(\sigma_j)}_g\)\right]\cdot (f_n)_g^{(1)}\cdot \Delta g
				\\
				= & \sum_{k=0}^{n-2}(\Delta g)^k\(\sum_{\substack{\sigma\in F_{n,k+1}\\\sigma_n=0}}\prod_{j=1}^{n}(f_j)^{(\sigma_j)}_g\)+\(\prod_{j=1}^{n-1}f_k^{(0)}\)\cdot (f_n)_g^{(1)} +\sum_{k=0}^{n-2}(\Delta g)^{k+1}\(\sum_{\substack{\sigma\in F_ {n,k+2}\\ \sigma_n=1}}\prod_{j=1}^{n}(f_j)^{(\sigma_j)}_g\)
				\\
				= & \sum_{k=0}^{n-2}(\Delta g)^k\(\sum_{\substack{\sigma\in F_{n,k+1}\\\sigma_n=0}}\prod_{j=1}^{n}(f_j)^{(\sigma_j)}_g\)+\(\prod_{j=1}^{n-1}f_k^{(0)}\)\cdot (f_n)_g^{(1)} +\sum_{k=1}^{n-1}(\Delta g)^{k}\(\sum_{\substack{\sigma\in F_ {n,k+1}\\ \sigma_n=1}}\prod_{j=1}^{n}(f_j)^{(\sigma_j)}_g\)
				\\
				= & \sum_{k=0}^{n-2}(\Delta g)^k\(\sum_{\substack{\sigma\in F_{n,k+1}\\\sigma_n=0}}\prod_{j=1}^{n}(f_j)^{(\sigma_j)}_g\) +\sum_{k=0}^{n-1}(\Delta g)^{k}\(\sum_{\substack{\sigma\in F_ {n,k+1}\\ \sigma_n=1}}\prod_{j=1}^{n}(f_j)^{(\sigma_j)}_g\)\\
				= & \sum_{k=0}^{n-1}(\Delta g)^k\(\sum_{\sigma\in F_{n,k+1}}\prod_{j=1}^{n}(f_j)^{(\sigma_j)}_g\),
			\end{align*}
			where in the last equality we have use the fact that $F_{n,n}=\{(1,\dots,1)\}$ and thus $\sigma_n=1$ for $\sigma\in F_{n,n}$.
		\end{proof}
		\begin{rem}
			Given that $(f_j)'_g(t^*)=(f_j)'_g(t)$, $j=1,2,\dots,n$, $t\in[a,b]$, we can write~\eqref{gderivativenproduct} as
			\[\(\prod_{j=1}^nf_j\)_g'(t)=\sum_{k=0}^{n-1}(\Delta g(t^*))^k\(\sum_{\sigma\in F_{n,k+1}^*}\prod_{j=1}^n(f_j)^{(\sigma_j)}_g(t)\),\]
			where $F_{n,k}^*:=\{\sigma\in \{*,1\}^n\ :\ \left\lvert \sigma\right\rvert=k\}$ with $\left\lvert (\sigma_1,\dots,\sigma_n)\right\rvert:=\operatorname{card}\{k\in\{1,\dots, n\}\ :\ \sigma_k=1\}$ and $f^{(*)}_g(t)\equiv f(t^*)$. Observe that, in this form, for $n=2$, we recover~\eqref{deltaproducto}.
		\end{rem}

		In this last part of the section, we shall focus on the concept of $g$-continuity, as presented in \cite[Definition 3.1]{FriLo17} which we include below. 

%

		\begin{dfn}[$g$-continuous function] A function $f:D\subset\mathbb R \to {\mathbb F}$ is
			\emph{$g$-continuous} at a point $t\in D$,
			or \emph{continuous with respect to $g$} at $t$, if for every $\varepsilon>0$, there exists $\delta>0$ such that 
			\[\left\lvert f(t)-f(s)\right\rvert<\varepsilon,\quad\mbox{ for every $s \in D$ such that $\left\lvert g(t)-g(s)\right\rvert<\delta$};\] 
			otherwise, we say that $f$ is \emph{$g$-discontinuous} at $t$.
			If $f$ is $g$-continuous at every point $t\in D$, we say that $f$ is $g$-continuous on $D$.
		\end{dfn}
		\begin{rem}
			In \cite[Remark 3.10]{MaTo19}, the authors presented the concept of lateral $g$-continuity. We say that $f$ is \emph{$g$-continuous from the left} at $t$ if for every $\varepsilon>0$, there exists $\delta>0$ such that 
			\[\left\lvert f(t)-f(s)\right\rvert<\varepsilon,\quad\mbox{ for every $s \in D$, $s\le t,$ such that $0\le g(t)-g(s)<\delta$}.\]
			Similarly, $f$ is \emph{$g$-continuous from the right} at $t$ if for every $\varepsilon>0$, there exists $\delta>0$ such that 
			\[\left\lvert f(t)-f(s)\right\rvert<\varepsilon,\quad\mbox{ for every $s \in D$, $s\ge t$, such that $0\le g(s)-g(t)<\delta$}.\]
			Naturally, $f$ is $g$-continuous at $t$ if and only if $f$ is $g$-continuous from the left and right at $t$.
		\end{rem}

		\begin{rem}
			It is important to note that, as indicated in \cite[Section 3]{FriLo17}, $g$-continuity can be understood as continuity between pseudometric spaces, which are sequential spaces (see \cite{Arkh}). Thus, we have that $f$ is $g$-continuous at $t\subset D$ if and only if 
			\begin{equation}\label{gseqcont}
				f(t_n)\to f(t),\quad\mbox{for every sequence } \{t_n\}_{n\in\mathbb N}\subset D \mbox{ such that }g(t_n)\to g(t).
			\end{equation}
		\end{rem}

		The following result,\cite[Proposition~3.2]{FriLo17}, describes some properties of $g$-continuous functions that are revelant for the work ahead.
		\begin{pro}\label{proreg} 
			If $f:[a, b] \to \mathbb{F}$ is $g$-continuous on $[a, b],$ then:
			\begin{itemize}
				\item $f$ is continuous from the left at every $t \in(a, b]$;
				\item if $g$ is continuous at $t \in[a, b),$ then so is $f$;
				\item if $g$ is constant on some $[\alpha, \beta] \subset[a, b],$ then so is $f$.
			\end{itemize}
		\end{pro}

		 In \cite{FriLo17} the reader can find more information on the properties of $g$-continuous functions. In the following, we present some results that, to the best of our knowledge, are not available in the current literature.

		\begin{lem}\label{lemdiv}
			Let $f,h:D\subset \mathbb R\to\bF$ be $g$-con\-tin\-u\-ous at $t\in[a,b]$ and $h(t)\ne 0$. Then $f/h$ is $g$-con\-tin\-u\-ous at $t$.
		\end{lem}
		\begin{proof}
			Let $t\in D$ and $\e>0$. Since $h$ is $g$-con\-tin\-u\-ous at $t$, there exists $\d_1\in\bR^+$ such that 
			\[\left\lvert h(s)-h(t)\right\rvert<\left\lvert h(t)\right\rvert/2,\quad \mbox{for }s\in D \mbox{ such that } \left\lvert g(s)-g(t)\right\rvert<\d_1.\] 
			Hence, $\left\lvert h(t)\right\rvert/2<h(s)<3\left\lvert h(t)\right\rvert/2$ for every $s\in D$ such that $\left\lvert g(s)-g(t)\right\rvert<\d_1$. 
			Now, since $f,h$ are $g$-continuous and $h(t)\ne 0$, we can find $\d\in(0,\d_1)$ such that \[\left\lvert f(s)-f(t)\right\rvert,\left\lvert h(s)-h(t)\right\rvert<\frac{\left\lvert h(t)\right\rvert^2}{2(\left\lvert h(t)\right\rvert+\left\lvert f(t)\right\rvert)}\e,\quad \mbox{for }s\in D\mbox{ such that } \left\lvert g(s)-g(t)\right\rvert<\d.\] 
			Thus, for every $s\in D$ such that $\left\lvert g(s)-g(t)\right\rvert<\d$,
			\[\left\lvert \frac{f(s)}{h(s)}-\frac{f(t)}{h(t)}\right\rvert= \frac{\left\lvert f(s)h(t)-f(t)h(s)\right\rvert}{\left\lvert h(t)h(s)\right\rvert}\le\frac{\left\lvert f(s)-f(t)\right\rvert\left\lvert h(t)\right\rvert+\left\lvert f(t)\right\rvert\left\lvert h(s)-h(t)\right\rvert}{\frac{1}{2}\left\lvert h(t)\right\rvert^2} \le \e,
			\]
			that is, $f/h$ is $g$-con\-tin\-u\-ous at $t$.
		\end{proof}

		\begin{lem}\label{tc}
			Let $f:D\subset \mathbb R\to\mathbb F$ and $t_1,t_2\in D$ be such that $g(t_1)=g(t_2)$. Then, $f$ is $g$-continuous at $t_1$ if and only if $f$ is $g$-continuous at $t_2$. Furthermore, in that case, we have that $f(t_1)=f(t_2)$.
%
	\end{lem}
	\begin{proof}
		Let $t_1,t_2\in D$ be such that $g(t_1)=g(t_2)$. Observe that it is enough to prove that if $f$ is $g$-continuoous at $t_1$ then $f(t_1)=f(t_2)$ as, in that case, the equivalence for the $g$-continuity follows from the definition.

		 Since $f$ is $g$-continuous at $t_1$, for every $\varepsilon>0$, there exists $\d>0$ such that if $s\in D$ satisfies $\left\lvert g(s)-g(t_1)\right\rvert<\d$, then $\left\lvert f(s)-f(t_1)\right\rvert<\e$. Hence, given that $g(t_1)=g(t_2)$ we have that $\left\lvert f(t_2)-f(t_1)\right\rvert<\e$ for every $\varepsilon>0$ so, necessarily, $f(t_1)=f(t_2)$, which finishes the proof.
%
	\end{proof}
%
%


		It is important to note that, as presented in \cite[Example~3.3]{FriLo17}, $g$-continuous functions need not be regulated or even locally bounded. With this idea in mind, given $a,b\in\mathbb R$, $a<b$, we define
		\[
		\mathcal{BC}_g([a,b],\mathbb F)=\left\{f:[a,b]\to\mathbb F: f\mbox{ bounded and $g$-continuous on }[a,b]\right\}.
		\]
		Similarly, 	\cite[Example~3.23]{MarquezTesis} shows that, unlike in the usual setting, $g$-differentiability does not imply $g$-continuity, which lead to the following definition.

		\begin{dfn}[{\cite[Definition 3.12]{Fernandez2021}}]
			Given $a,b\in\mathbb R$, $a<b$, such that $a\notin N_g^-\cup D_g$ and $b\notin C_g\cup N_g^+\cup D_g$, we define
			\[
				\mathcal{BC}_g^1([a,b],{\mathbb F}):=\{f\in \mathcal{BC}_g([a,b],{\mathbb F}): f'_g \in \mathcal{BC}_g([a,b],{\mathbb F})\}.
			\]
		\end{dfn}

	\section{Some properties of $\Delta g$}

	We now turn our attention to the study of the function $\Delta g:\mathbb R\to\mathbb R$. 
	We start with a simple result showing 
	 that it is a regulated function, which is enough to ensure measurability with respect to $g$.
	\begin{pro}\label{regulated} 
		 For each $t\in\mathbb R$, we have that
		\begin{equation}\label{Deltalimit}
			\lim_{s\to t^-}\Delta g(t)=\lim_{s\to t^+}\Delta g(t)=0.
		\end{equation}
		In particular, $\Delta g$ is a regulated function, Borel-measurable and $g$-measurable.
	\end{pro}
	\begin{proof}
	In order to show that~\eqref{Deltalimit} holds, let $t\in\mathbb R$. We first show that $\Delta g(t^+)=0$. 

	If $t\not\in (D_g\cap(t,+\infty))'$, then there exists $r>0$ such that $D_g\cap(t,+\infty)\cap(t-r,t+r)=D_g\cap(t,t+r)=\emptyset$, so $g$ is continuous on $(t,t+r)$, which implies that $\Delta g=0$ on that same interval and the result follows. Otherwise, we have that $t\in (D_g\cap(t,+\infty))'$ and we need to show that $\Delta g(t^+)=0$. This happens if and only if
	\[\lim_{\substack{s\to t^+\\ s\in D_g}}\Delta g(t)=\lim_{\substack{s\to t^+\\ s\not\in D_g}}\Delta g(t)=0,\]
	where the second limit can be taken because $D_g$ is countable, so $\mathbb R\backslash D_g$ is dense in $\mathbb R$.
	However, noting that $\Delta g=0$ on $\mathbb R\backslash D_g$, we have that the second of these limits is null, so it is enough to show that
	\begin{equation*}
		\lim_{\substack{s\to t^+\\ s\in D_g}}\Delta g(t)=0.
	\end{equation*}

	Let $\{t_n\}_{n\in\mathbb N}$ be a sequence in $D_g\cap (t,+\infty)$ converging to $t$. Then, there exists $N\in\mathbb N$ such that $t_n\in (t,t+1)$ for all $n\in\mathbb N$ such that $n\ge N$. Hence,
	\[0<\sum_{n=N}^\infty \Delta g(t_n)\le \sum_{s\in (t,t+1)\cap D_g} \Delta g(s)= \int_{(t,t+1)\cap D_g} \dif g(s)\le \int_{[t,t+1)}\dif g(s)=\mu_g([t,t+1))<+\infty.\]
	As a consequence, the series $\sum_{n\in\mathbb N}\Delta g(t_n)$ is convergent, which guarantees that the sequence $\{\Delta g(t_n)\}_{n\in\mathbb N}$ must converge to $0$. Since $\{t_n\}_{n\in\mathbb N}$ was arbitrarily chosen, $\Delta g(t^+)=0$.

	The reasoning for $\Delta g(t^-)$ is analogous and we omit it. Observe that this is enough to guarantee that $\Delta g$ is regulated, which in turn makes it Borel measurable and, thus, $g$-measurable.
	\end{proof}

Taking into account the previous result it is only natural to wonder what these limits imply with regard to the $g$-continuity of $\Delta g$. This is covered in the next result.

%
%

\begin{pro}\label{proddg}
	Consider the set
	\begin{equation}\label{Ag}
		A_g=\{t\in\mathbb R: \sup\{s\in [a,b]\ :\ g(s)=g(t)\}\in D_g\}.
	\end{equation}
	The map $\Delta g$ is $g$-continuous at every point of 
$\mathbb R\backslash A_g$
	 and $g$-discontinuous at every point of $A_g$.
	\end{pro}
\begin{proof}
	Throughout this proof, given $t\in\mathbb R$, we denote
	\begin{equation}\label{hatt}
		\widehat t=\sup\{s\in [a,b]\ :\ g(s)=g(t)\}.
	\end{equation}

	First, we prove that $\Delta g$ is $g$-discontinuous at every point of $D_g$. To that end, let $t\in D_g$. Then, either~\eqref{cond1} holds or
	\begin{equation}\label{cond3}
		g(s)=g(t),\quad\mbox{for all }s\in[t_0,t]\mbox{ for some }t_0<t.
	\end{equation}

	Suppose~\eqref{cond3} holds. Then $\Delta g(t_0)=g(t_0^+)-g(t_0)=g(t)-g(t)=0$, while $\Delta g(t)> 0$, so for every $\d\in\bR^+$ we have that $\left\lvert g(t_0)-g(t)\right\rvert<\delta$ but $\left\lvert \Delta g(t_0)-\Delta g(t)\right\rvert=\Delta g(t)> \Delta g(t)/2$, so $\Delta g$ is not $g$-con\-tin\-u\-ous at $t$.

	If~\eqref{cond1} holds instead, since $\lim_{s\to t^-}\Delta g(s)=0$, there exists $\rho>0$ such that \[\Delta g(s)<\frac{\Delta g(t)}{2},\quad s\in\bR,\ 0< t-s<\rho.\] 
	Since~\eqref{cond1} holds, so we can find $t_1<t$ such that $t-t_1<\rho$ and $g(t_1)<g(t)$. Let $\d=g(t)-g(t_1)$. Observe that, if $0< g(t)-g(s)<\d$, since $g$ is nondecreasing, then $0<t-s<t-t_1<\rho$, and so $\Delta g(s)<\Delta g(t)/2$. But this implies that, if $0< g(t)-g(s)<\d$, then
	\[\left\lvert \Delta g(s)-\Delta g(t)\right\rvert\ge \Delta g(t)-\Delta g(s)>\Delta g(t)-\frac{\Delta g(t)}{2}=\frac{\Delta g(t)}{2},\]
	and, therefore, $\Delta g$ is not continuous at $t$.

	Now, let $t\in A_g\backslash D_g$. Since $g$ is left-continuous, we have that $g(t)=g(\widehat t)$ so it follows from Lemma~\ref{tc} that $\Delta g$ cannot be $g$-continuous at $t$ since it is $g$-discontinuous at $\widehat t$ as $\widehat t\in D_g$.

	Now, we shall study the $g$-continuity of $\Delta g$ in the remaining points. Let $t\in \mathbb R\backslash A_g$. In that case, $t\not\in D_g$ so $\Delta g(t)=0$. We claim that 
	\begin{align}
		\Delta g&\mbox{ is $g$-continuous from the left at $t$ if~\eqref{cond1} holds};\label{Deltaleftcont}\\
		\Delta g&\mbox{ is $g$-continuous from the right at $t$ if~\eqref{cond2} holds}.\label{Deltarightcont}
	\end{align}

	Suppose that~\eqref{cond1} holds and let $\varepsilon>0$. Since $\lim_{s\to t^-}\Delta g(s)=0$ and $\Delta g(t)=0$, there exists $\rho>0$ such that $\Delta g(s)<\varepsilon$, for all $s\in(t-\rho,t]$.
	Since~\eqref{cond1} holds, we can take $\delta=g(t)-g(t-\rho)>0$. If $s\le t$ is such that $0\le g(t)-g(s)<\delta$, since $g$ is nondecreasing, we necessarily have that $s\in(t-\rho, t]$ so
		\[\left\lvert \Delta g(s)-\Delta g(t)\right\rvert=\Delta g(s)<\varepsilon,\]
	which proves that $\Delta g$ is also $g$-continuous from the left at $t$.

	The proof of~\eqref{Deltarightcont} is analogous, and we omit it. Since~\eqref{Deltaleftcont}-\eqref{Deltarightcont} hold, it follows that $\Delta g$ is $g$-continuous if $t\not\in C_g\cup N_g\cup A_g$. Similarly,~\eqref{Deltaleftcont} guarantees that $\Delta g$ is $g$-continuous from the left if $t\in N_g^-\backslash A_g$, while~\eqref{Deltarightcont} ensures the $g$-continuity from the right if $t\in N_g^+\backslash A_g$. 

	Let us investigate what happens in the remaining cases, namely, when $t\in (C_g \cup N_g)\backslash A_g$.

	Suppose $t\in (N_g^-\cup C_g)\backslash A_g$ and let us show that $\Delta g$ is $g$-continuous from the right at $t$.
	Let $\varepsilon>0$ and consider $\widehat t$ as in~\eqref{hatt}.
	Observe that $g(s)=g(t)$ for all $s\in[t,\widehat t)$, which means $\Delta g(s)=0$ for all $s\in[t,\widehat t)$.
	Furthermore, since $t\not\in A_g$, we have that $\widehat t\not\in D_g$, so we necessarily have that $g(\widehat t)=g(t)$ and $\widehat t\in N_g^+$. In that case, $\Delta g$ is $g$-continuous from the right at $\widehat t$, which means that there exists $\delta>0$ such that
	\[\left\lvert \Delta g(s)\right\rvert=\left\lvert \Delta g(\widehat t)-\Delta g(s)\right\rvert<\varepsilon,\quad\mbox{ for each $s\ge \widehat t$ such that $0\le g(s)-g(\widehat t)<\delta$}.\]
	Now, if $s\ge t$ is such that $0\le g(s)-g(t)<\delta$, then either $s\in[t,\widehat t)$, or $s\ge \widehat t$ and $0\le g(s)-g(\widehat t)<\delta$. In either case, it follows that
	\[\left\lvert \Delta g(t)-\Delta g(s)\right\rvert=\left\lvert \Delta g(s)\right\rvert<\varepsilon,\]
	proving that $\Delta g$ is $g$-continuous from the right at $t$. In particular, this proves that $\Delta g$ is $g$-continuous if $t\in N_g^-$.

		Finally, suppose $t\in (N_g^+\cup C_g)\backslash A_g$ and let us show that $\Delta g$ is $g$-continuous from the left at $t$. Let $\varepsilon>0$ and consider
		\[\widetilde t=\inf\{s\in [a,b]\ :\ g(s)=g(t)\}.\]
		Observe that $g(s)=g(t)$ for all $s\in(\widetilde t,t]$, which means $\Delta g(s)=0$ for all $s\in(\widetilde t,t]$. Now, if $\widetilde t\not\in D_g$, then $g(\widetilde t)=g(t)$ and $\widetilde t\in N_g^-$. In that case, $\Delta g$ is $g$-continuous from the left at $\widetilde t$ so there exists $\delta>0$ such that
		\[\left\lvert \Delta g(s)\right\rvert=\left\lvert \Delta g(\widetilde t)-\Delta g(s)\right\rvert<\varepsilon,\quad\mbox{ for each $s\le \widehat t$ such that $0\le g(\widetilde t)-g(s)<\delta$}.\]
		Now, if $s\le t$ is such that $0\le g(t)-g(s)<\delta$, then either $s\in(\widetilde t,t]$, or $s\le \widetilde t$ and $0\le g(\widetilde t)-g(s)<\delta$. In either case, it follows that
		\[\left\lvert \Delta g(t)-\Delta g(s)\right\rvert=\left\lvert \Delta g(s)\right\rvert<\varepsilon.\]
		Otherwise, $\widetilde t\in D_g$, so taking $\delta\in(0,\Delta g(\widetilde t))$, it follows that if
		$s\le t$ is such that $0\le g(t)-g(s)<\delta$, then $s\in(\widetilde t,t]$, so
		\[\left\lvert \Delta g(t)-\Delta g(s)\right\rvert=0<\varepsilon,\]
		which finishes the proof. 
	\end{proof}

%

\begin{exa}~\label{exaa} To illustrate the set $A_g$, consider the function
represented in Figure~\ref{A}. Here $A_g=\{-1\}\cup[0,2]$.
\begin{figure}[h]
	\centering
	\includegraphics[width=.5\textwidth]{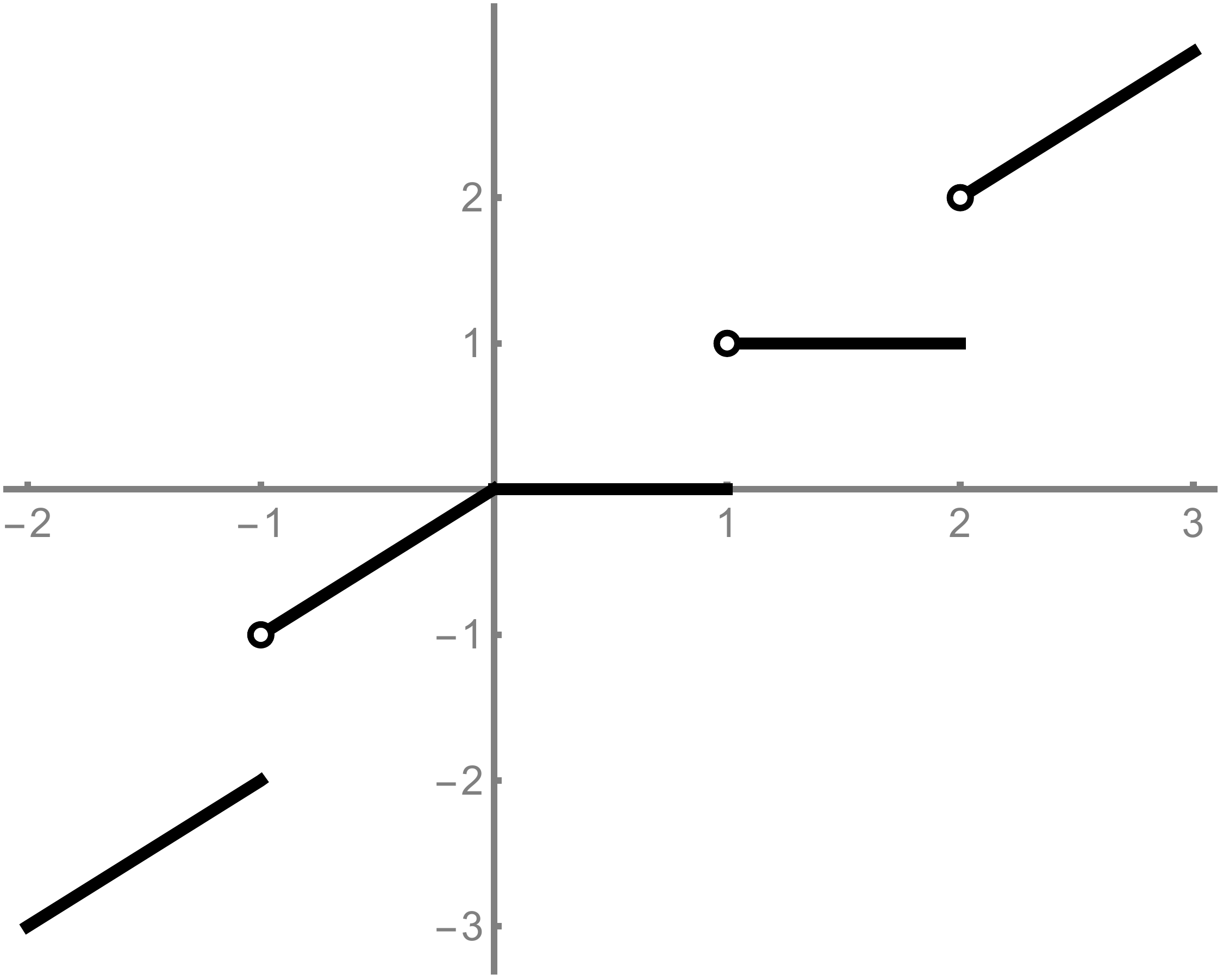}
	\caption{The graph of function $g$ in Example~\ref{exaa}.}
	\label{A}
\end{figure}
\end{exa}

	We now consider the $g$-differentiability of $\Delta g$. In order to do so, we present the following more general result from which we can deduce some information about the differentiability in the Stieltjes sense of $\Delta g$.
\begin{pro}\label{propfDeltadif}
	Let $f:[a,b]\to\mathbb F$ be a map and $h:[a,b]\to\mathbb F$ be defined as $h(t)=f(t)\Delta g(t)$.
		Consider the sets
		\begin{align}
			D_1&=\{t\in[a,b]\cap N_g^-: t\in (D_g\cap[a,t))'\},\label{defD1}\\
			D_2&=\{t\in[a,b]\cap N_g^+: t\in (D_g\cap(t,b])'\},\label{defD2}\\
			D_3&=\{t\in[a,b]\backslash (N_g\cup D_g)
			: t\in (D_g\cap[a,b])'\}.\label{defD3}
		\end{align}		
	For $t\in[a,b]$, and denoting by $t^*$ the corresponding point as in~\eqref{tstar}, we have the following properties:
	\begin{itemize}
		\item[\textup{1.}] If $t^*\in D_1\cup D_2\cup D_3$, then $h$ is $g$-differentiable at $t$ if and only if
		\begin{equation}\label{condfDeltaderivable1}
			\lim_{\substack{s\to t^*\\ s\in D_g}}\frac{f(s)\Delta g(s)}{g(s)-g(t)}=0,
		\end{equation}
		where we might be considering the corresponding side limit according to the definition of Stieltjes derivative, see \textup{Remark~\ref{remNgderivative}}.
		\item[\textup{2.}] If $t^*\in D_g\cap (D_g\cap(t,b])'$, then $h$ is $g$-differentiable at $t$ if and only if
		\begin{equation*}
			\lim_{\substack{s\to t^{*+}\\ s\in D_g}}f(s)\Delta g(s)=0.
		\end{equation*}
		\item[\textup{3.}] In any other case, $h$ is $g$-differentiable at $t$.
	\end{itemize}
	Moreover, if $h$ is $g$-differentiable at $t$, then
	\begin{equation}\label{hgder}
		h'_g(t)=-f(t^*)\irchi_{D_g}(t^*).
	\end{equation}
\end{pro}
\begin{proof}
We will first prove 1-3 in order for $t\in[a,b]\backslash C_g$ while showing that~\eqref{hgder} holds for each case. 

Let $t\in[a,b]\backslash C_g$. Observe that, in that case, $t^*=t$.

First, suppose $t\in D_1\cup D_2\cup D_3$.
	Observe that, in that case, $t\not\in D_g$ so $\Delta g(t)=0$.
Now, if $h$ is $g$-differentiable at $t$, by definition, \[h'_g(t)=\lim_{s \to t}\frac{h(s)-h(t)}{g(s)-g(t)}=\lim_{s \to t}\frac{f(s)\Delta g(s)}{g(s)-g(t)}.\]
Observe that, since such limit exists, the following limits exist and are equal and, since $\Delta g=0$ in $[a,b]\backslash D_g$, we must have that
\[\lim_{\substack{s\to t\\ s\in D_g}}\frac{f(s)\Delta g(s)}{g(s)-g(t)}=\lim_{\substack{s\to t\\ s\not\in D_g}}\frac{f(s)\Delta g(s)}{g(s)-g(t)}=0,\]
where the second limit can be taken because $D_g$ is countable, so $[a,b]\backslash D_g$ is dense in $[a,b]$. This demonstrates that if $h$ is $g$-differentiable at $t$ then~\eqref{condfDeltaderivable1} holds. Furthermore, this also shows that~\eqref{hgder} holds in this case. 
Conversely, assume that~\eqref{condfDeltaderivable1} holds. It is clear now that we need to show that
\[\lim_{s\to t}\frac{f(s)\Delta g(s)}{g(s)-g(t)}=0.\]
Once again, such limit exists provided that
\[\lim_{\substack{s\to t\\ s\in D_g}}\frac{f(s)\Delta g(s)}{g(s)-g(t)}=\lim_{\substack{s\to t\\ s\not\in D_g}}\frac{f(s)\Delta g(s)}{g(s)-g(t)}=0.\]
Note that this is the case thanks to condition~\eqref{condfDeltaderivable1} and the fact that $\Delta g=0$ in $[a,b]\backslash D_g$.

Now, suppose $t\in D_g\cap (D_g\cap(t,b])'$. In that case, the hypotheses guarantee that $t\not=b$ and $h=0$ on $(t,b]\backslash D_g$, which is dense on $(t,b]$ as $D_g$ is countable. Hence, we have that
\[\lim_{\substack{s\to t^+\\ s\in D_g}}h(s)=0.\]
Now, Remark~\ref{remDgCg} ensures that $h$ is $g$-differentiable at $t$ if and only if $h(t^+)$ exists which, at the same time, exists if and only if
\[\lim_{\substack{s\to t^+\\ s\not\in D_g}}h(s)=0.\]
Observe that this proves 2 and, furthermore, equality~\eqref{hgder} as, under these circumstances, $h(t^+)=0$ and
\[h'_g(t)=\frac{h(t^+)-h(t)}{\Delta g(t)}=\frac{-f(t)\Delta g(t)}{\Delta g(t)}=-f(t)=-f(t^*)\irchi_{D_g}(t^*).\]

Finally, for 3, we first consider the case $t\in[a,b]\backslash (D_g'\cup C_g)$. In that case, there is $\varepsilon>0$ such that $((t-\varepsilon,t+\varepsilon)\backslash\{t\})\cap D_g=\emptyset$, which implies that $g$ is continuous in $(t-\varepsilon,t+\varepsilon)\backslash\{t\}$ and, in particular, $\Delta g=0$ in that set. This means that $h=0$ on $(t-\varepsilon,t+\varepsilon)\backslash\{t\}$. Therefore, if $t\not\in D_g$,
\[\lim_{s \to t}\frac{h(s)-h(t)}{g(s)-g(t)}=\lim_{s \to t}\frac{h(s)}{g(s)-g(t)}=0,\]
which means that $h$ is $g$-differentiable at $t$ and $h'_g(t)=0=-f(t^*)\irchi_{D_g}(t^*)$. Otherwise, we have that $t\in D_g$ and, in that case, it follows that $h(t^+)=0$. Remark~\ref{remDgCg} guarantees that $h$ is $g$-differentiable at $t$ and
\[h'_g(t)=\frac{h(t^+)-h(t)}{\Delta g(t)}=\frac{-f(t)\Delta g(t)}{\Delta g(t)}=-f(t)=-f(t^*)\irchi_{D_g}(t^*).\]

Hence, all that is left to do to prove 3 for $t\in [a,b]\backslash C_g$ is to show that $h$ is $g$-differentiable if $t\in D_g'$ but does not satisfy the conditions in 1 and 2, that is, if $t\in D_g\cap(D_g\cap[a,b])'$ and $t\not\in D_g\cap (D_g\cap(t,b])'$. In this setting, there exists $r>0$ such that $(t,t+r)\cap D_g=\emptyset$, which guarantees that $\Delta g=0$ on $(t,t+r)$. This implies that $h=0$ on $(t,t+r)$ and, thus, $h(t^+)=0$. As a consequence, $h$ is $g$-differentiable at $t$ and 
\[h'_g(t)=\frac{h(t^+)-h(t)}{\Delta g(t)}=\frac{-f(t)\Delta g(t)}{\Delta g(t)}=-f(t)=-f(t^*)\irchi_{D_g}(t^*).\]

Lastly, we prove the result for $t\in[a,b]\cap C_g$. Let $t\in(a_n,b_n)\subset [a,b]\cap C_g$ for some $a_n,b_n$ in~\eqref{Cgdisj}. In that case, $t^*=b_n$ and, as pointed out in Remark~\ref{remDgCg}, $h$ is $g$-differentiable at $t$ if and only if $h$ is $g$-differentiable at $t^*$.
Hence, 1, 2 and 3 follow from the previous cases. Furthermore, noting that $t^*\in D_g$ if and only if $b_n\in D_g$ and using~\eqref{hgder}, we have that
\[h'_g(t)=h'_g(b_n)=-f(b_n^*)\irchi_{D_g}(b_n^*)=-f(b_n)\irchi_{D_g}(b_n)=-f(t^*)\irchi_{D_g}(t^*),\]
that is,~\eqref{condfDeltaderivable1} also holds for these cases.
\end{proof}
	\begin{rem}To visualize the sets $D_1$, $D_2$ and $D_3$, we provide an illustration in Figure~\ref{Ds}.
\end{rem}
\begin{figure}[h]
	\centering
	\includegraphics[width=.29\textwidth]{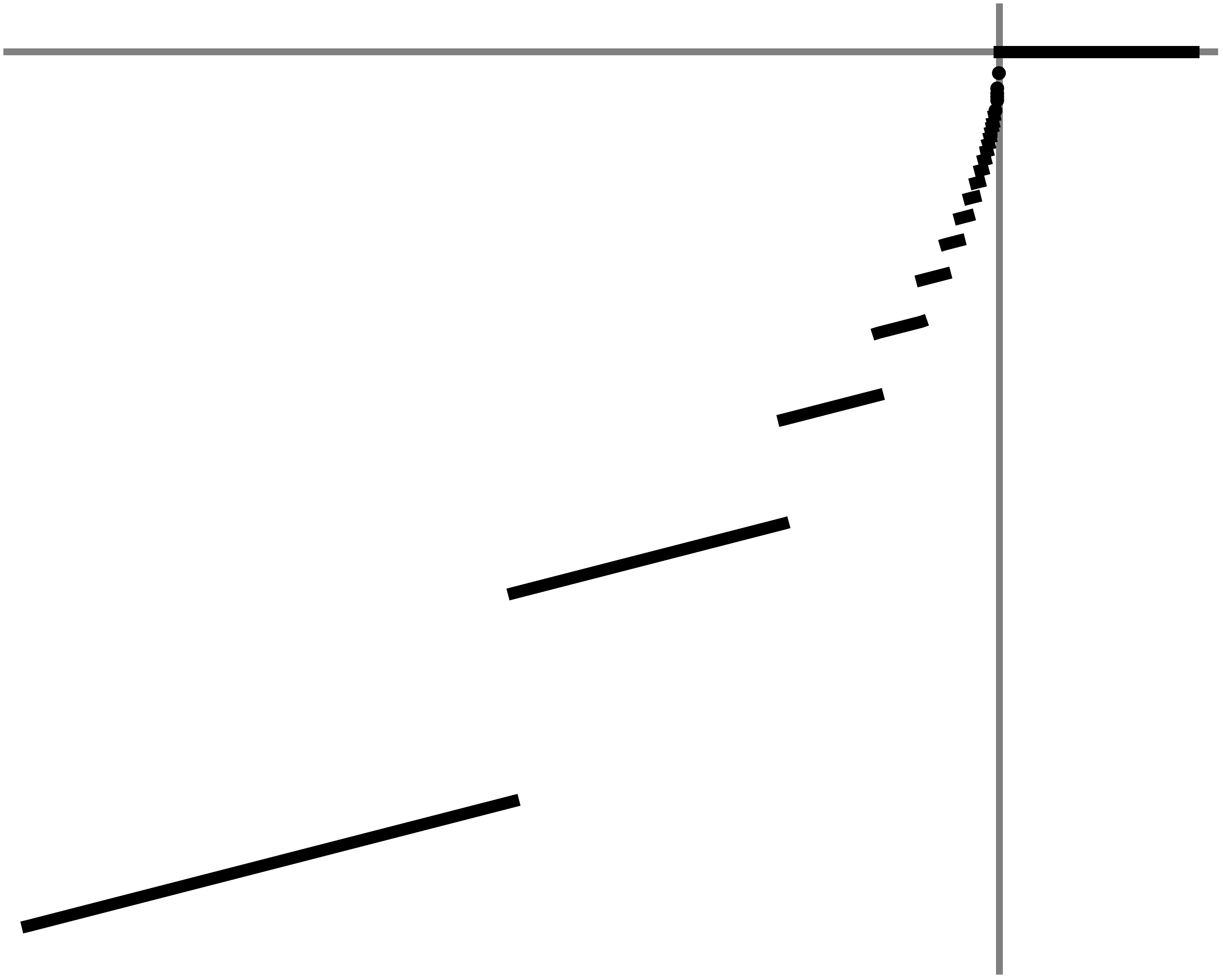}\quad \includegraphics[width=.29\textwidth]{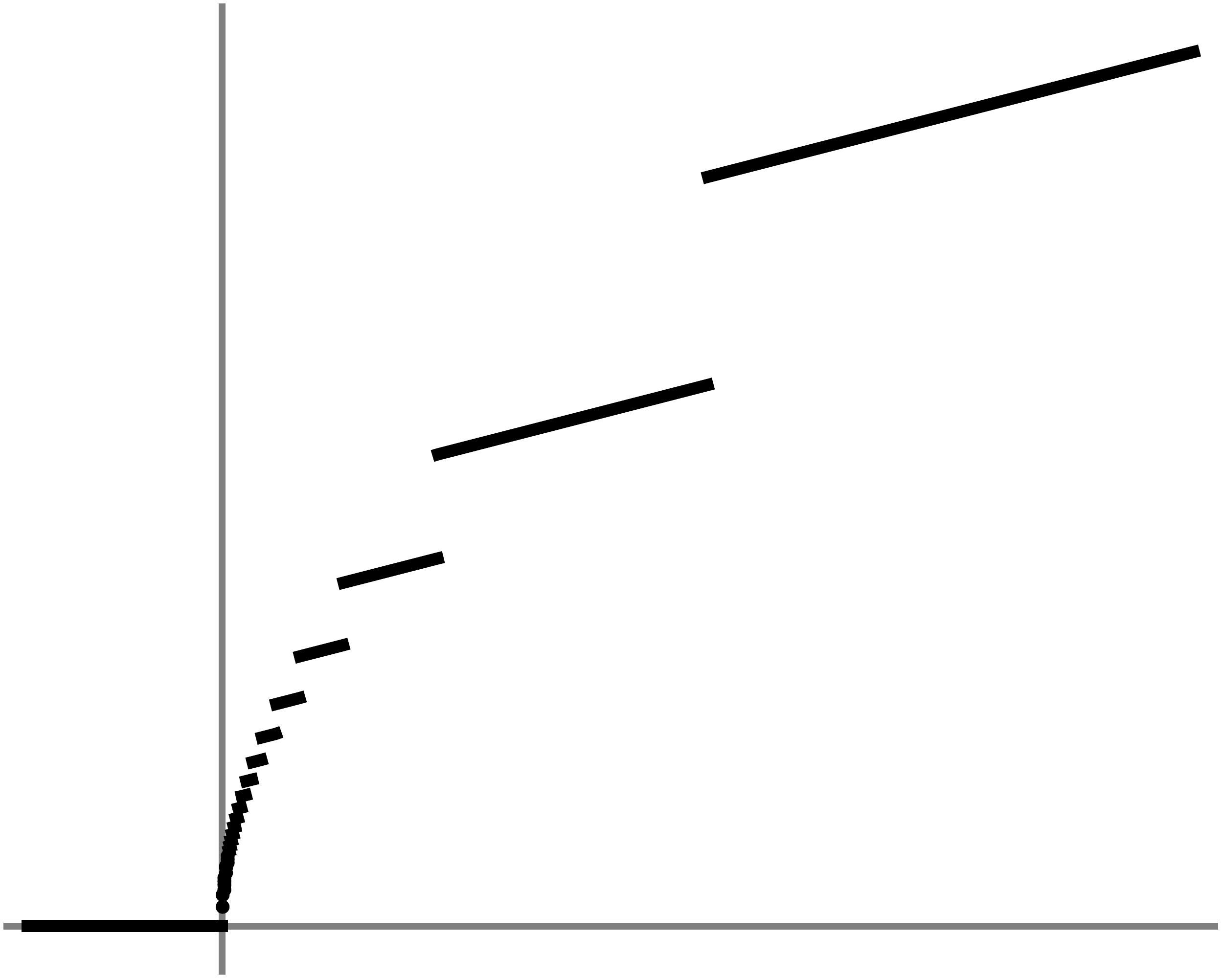}\quad \includegraphics[width=.29\textwidth]{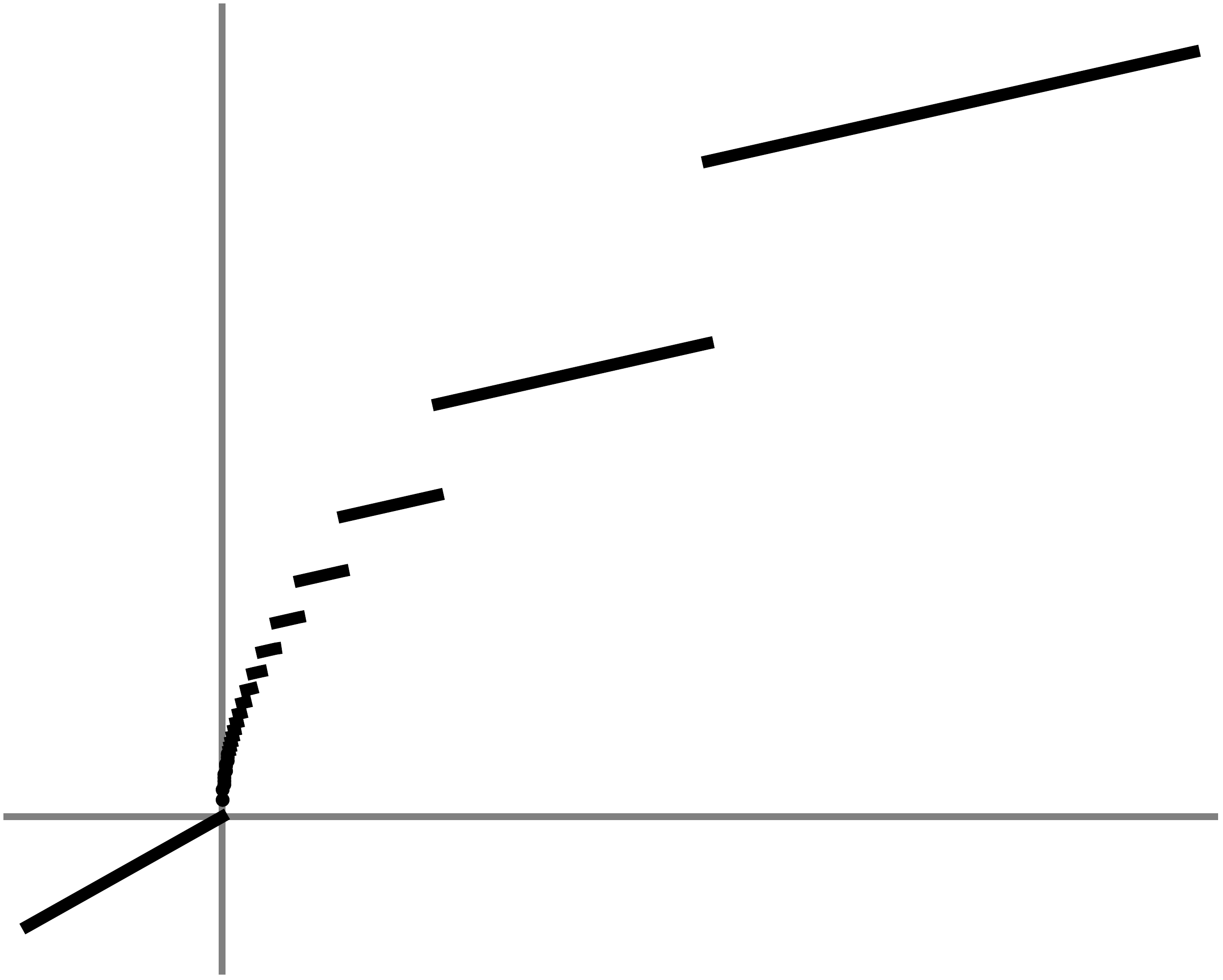}\quad 
	\caption{From left to right, representation of derivators for which zero belongs to $D_1$, $D_2$ and $D_3$ respectively.}
	\label{Ds}
\end{figure}

Combining Propositions~\ref{regulated} and~\ref{propfDeltadif} we can obtain the following result about the existence of the Stieltjes derivative of $\Delta g$.
\begin{cor}\label{differentiableDeltag}
	Consider the
	 sets $D_1, D_2, D_3$ in~\eqref{defD1}-\eqref{defD3} and the 
	restriction $\Delta g|_{[a,b]}$. For $t\in[a,b]$, and denoting by $t^*$ the corresponding point in~\eqref{tstar}, we have the following properties:
	\begin{itemize}
		\item[\textup{1.}] If $t^*\in D_1\cup D_2\cup D_3$
		then $\Delta g|_{[a,b]}$ is $g$-differentiable at $t$ if and only if
		\begin{equation}\label{pointwisecondDeltaderivable}
			\lim_{\substack{s\to t^*\\ s\in D_g}}\frac{\Delta g|_{[a,b]}(s)}{g(s)-g(t)}=0,
		\end{equation}
	 where we might be considering the corresponding side limit according to the definition of Stieltjes derivative, see \textup{Remark~\ref{remNgderivative}}.
		\item[\textup{2.}] In any other case, $\Delta g|_{[a,b]}$ is $g$-differentiable at $t$.
	\end{itemize}
	Furthermore, if $\Delta g|_{[a,b]}$ is $g$-differentiable at $t$, then
	\begin{equation}\label{Deltagder}
		\left(\Delta g|_{[a,b]}\right)'_g(t)=-\irchi_{D_g}(t^*).
	\end{equation}
	In particular, $\Delta g|_{[a,b]}$ is $g$-differentiable on $[a,b]$ and~\eqref{Deltagder} holds on $[a,b]$ if
	\begin{equation}\label{condDeltaderivable}
		\lim_{\substack{s\to t\\ s\in D_g}}\frac{\Delta g|_{[a,b]}(s)}{g(s)-g(t)}=0,\quad\mbox{for all }
		t\in D_1\cup D_2\cup D_3,
	\end{equation}
 with the same consideration for the side limits, see \textup{Remark~\ref{remNgderivative}}.
\end{cor}

\begin{rem}\label{remarkDeltadiff}
	Note condition~\eqref{condDeltaderivable} becomes vacuous when $[a,b]\cap D_g$ is closed, as in that case $([a,b]\cap D_g)'\subset [a,b]\cap D_g$, so $D_1=D_2=D_3=\emptyset$.
	Furthermore, condition~\eqref{condDeltaderivable} can be satisfied when $([a,b]\cap D_g)'\backslash ([a,b]\cap D_g)\not=\emptyset$. Indeed, the function
	\[g(t)=t+\sum_{n=1}^\infty2^{-n}\irchi_{(1/n,+\infty)}(t),\quad t\in\mathbb R,\]
	represented in Figure~\ref{f1} 
	is nondecreasing, left-continuous with $D_g=\{1/n: n\in\mathbb N\}$ and
	\[
	\Delta g(t)=\left\{
	\begin{array}{ll}
		2^{-n},\quad & t=1/n \mbox{ for some }n\in\mathbb N,\\
		0,\quad & \mbox{otherwise.}
	\end{array}
	\right.
	\]
	\begin{figure}[h]
		\centering
		\includegraphics[width=\textwidth]{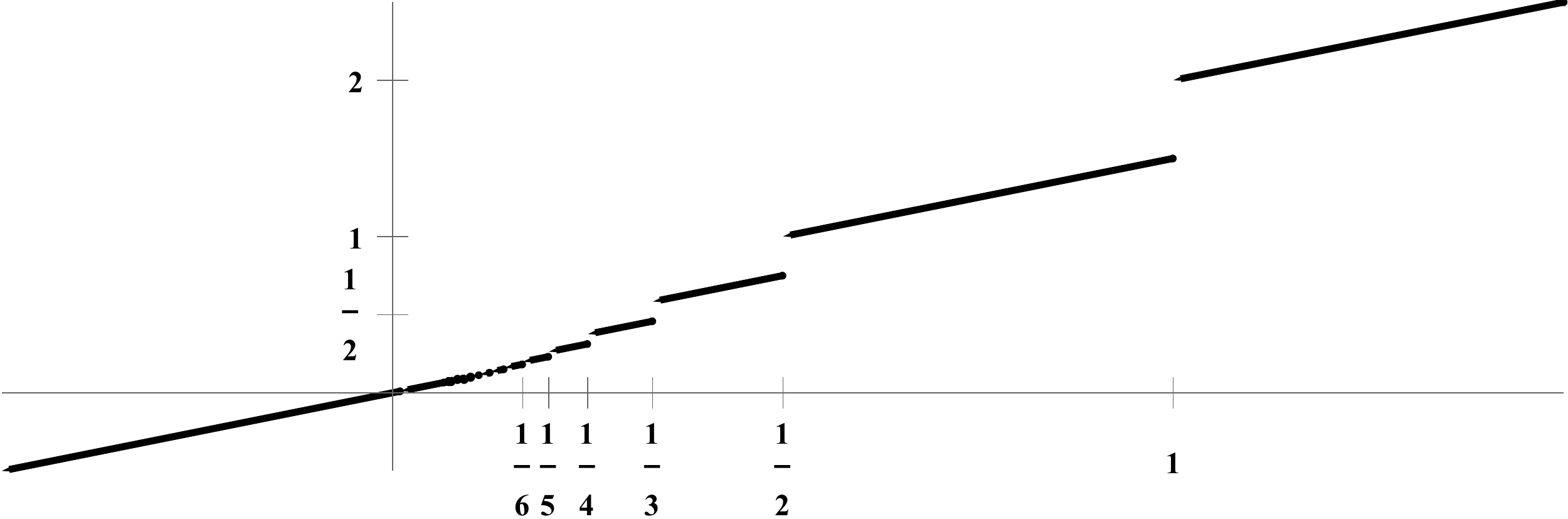}
		\caption{Graph of the function $g$ in Remark~\ref{remarkDeltadiff}.}
		\label{f1}
	\end{figure}

	Observe that $-1\not\in D_g$, $2\not\in D_g\cup C_g\cup N_g^+$ and
	$([-1,2]\cap D_g)'\backslash D_g=\{0\}$. Hence, we have that $D_1=D_2=\emptyset$ and $D_3=\{0\}$ and, furthermore,
	\begin{align*}
		\lim_{\substack{s\to 0\\ s\in D_g}}\frac{\Delta g|_{[-1,2]}(s)}{g(s)-g(0)}&=\lim_{n\to+\infty}\frac{\Delta g|_{[-1,2]}(1/n)}{g(1/n)}=\lim_{n\to+\infty}\frac{2^{-n}}{1/n+\sum_{k=1}^\infty2^{-k}\irchi_{(1/k,+\infty)}(1/n)}\\
		&=\lim_{n\to+\infty}\frac{2^{-n}}{1/n+\sum_{k=n+1}^\infty2^{-k}}=\lim_{n\to+\infty}\frac{2^{-n}}{1/n+2^{-n}}=\lim_{n\to+\infty}\frac{1}{2^n/n+1}=0,
	\end{align*}
	that is,~\eqref{condDeltaderivable} is satisfied on $[-1,2]$. As a consequence, we have that $\Delta g|_{[-1,2]}$ is $g$-differentiable and, in this case, $(\Delta g|_{[-1,2]})'_g(t)=\irchi_{D_g}(t)$, $t\in [-1,2]$. This example also shows that $(\Delta g|_{[-1,2]})'_g$ needs not be $g$-differentiable as $(\Delta g|_{[-1,2]})''_g(0)$ cannot exist since
	\[\lim_{\substack{s\to 0\\ s\in D_g}}\frac{(\Delta g|_{[-1,2]})'_g(s)-(\Delta g|_{[-1,2]})'_g(0)}{g(s)-g(0)}=\lim_{\substack{s\to 0\\ s\in D_g}}\frac{1-0}{g(s)-g(0)}=+\infty.\]

	Nevertheless, we should highlight that there are nondecreasing and left-continuous functions satisfying all the conditions but~\eqref{condDeltaderivable}. Indeed, it is enough to consider a small modification of the previous example, namely,
	\begin{equation}\label{gtilde}
		\widetilde g(t)=t\irchi_{(1,+\infty)}(t)+\sum_{n=1}^\infty2^{-n}\irchi_{(1/n,+\infty)}(t),\quad t\in\mathbb R,
	\end{equation}
	represented in Figure~\ref{f2}.
	\begin{figure}[h]
		\centering
		\includegraphics[width=\textwidth]{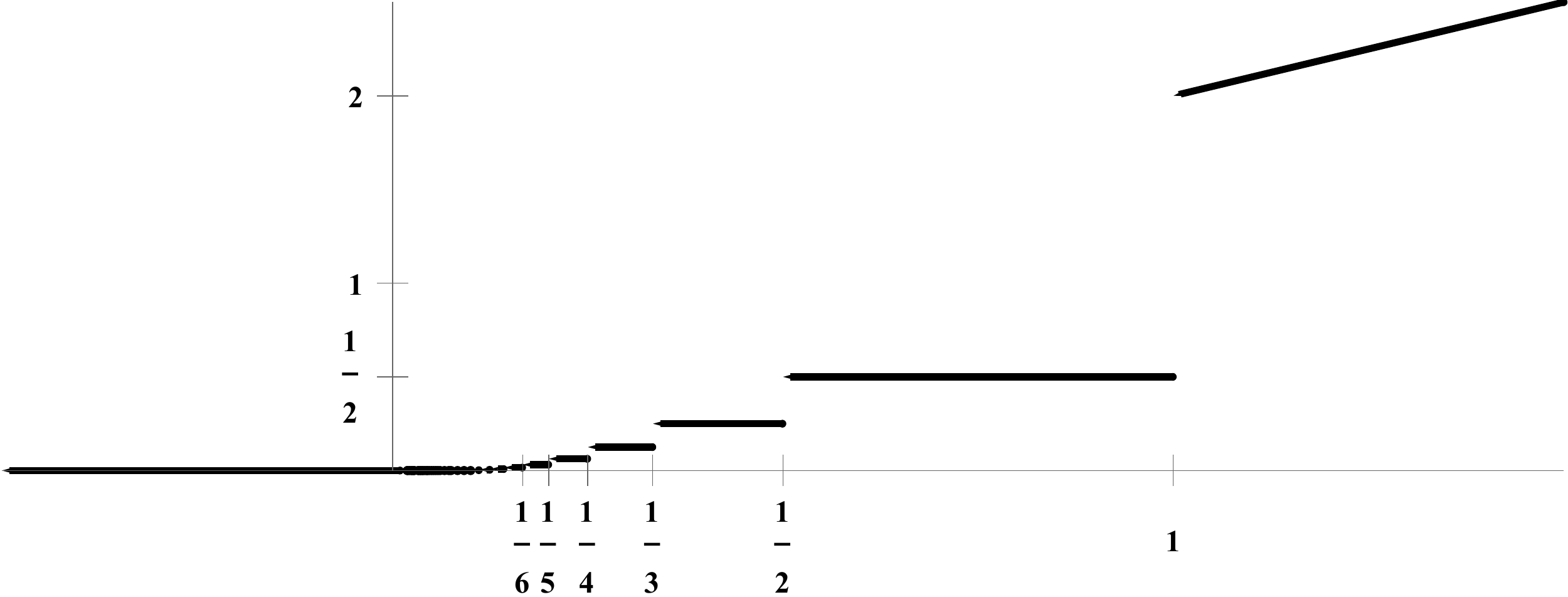}
		\caption{Graph of the function $\widetilde g$ in Remark~\ref{remarkDeltadiff}.}
		\label{f2}
	\end{figure}	

	Once again, this is a nondecreasing and left-continuous map such that $D_{\widetilde g}=D_g$ and $\Delta \widetilde g=\Delta g+\irchi_{\{1\}}$. Furthermore,
	\begin{equation}\label{Cgtilde}
		C_{\widetilde g}=(-\infty,0)\cup\bigcup_{n=1}^\infty \left(\frac{1}{n+1},\frac{1}{n}\right),
	\end{equation}
	so it follows that $-1\not\in D_{\widetilde g}\cup N_{\widetilde g}^-$, $2\not\in D_{\widetilde g}\cup C_{\widetilde g}\cup N_{\widetilde g}^+$.
	However, condition~\eqref{condDeltaderivable} cannot be satisfied as, in this case, $D_1=D_3=\emptyset$ and $D_2=\{0\}$, $\widetilde g(0)=0$
	and
	\[
	\lim_{n\to+\infty}\frac{\Delta \widetilde g|_{[-1,2]}(1/n)}{\widetilde g(1/n)}=\lim_{n\to+\infty}\frac{2^{-n}}{\sum_{k=1}^\infty2^{-k}\irchi_{(1/k,+\infty)}(1/n)}=\lim_{n\to+\infty}\frac{2^{-n}}{2^{-n}}=1.
	\]
\end{rem}

\section{Differentiating the product rule}
The aim of this section is to study the differentiability of the product of two differentiable functions beyond the first derivative, as this case is already covered by Proposition~\ref{PropStiDer}. 

Observe that Proposition~\ref{PropStiDer} provides us a good starting point for our research. Indeed, given that in~\eqref{deltaproducto} some of the maps involved are evaluated at $t^*$, we need to consider how this affects the Stieltjes differentiability of functions. With this idea in mind, we present the following result establishing some relations between a differentiable map and its corresponding counterpart evaluated at $t^*$.


\begin{pro}\label{differentiablef*}
	Let $f:[a,b]\to\mathbb F$ and define $f^*:[a,b]\to\mathbb F$ as $f^*(t)=f(t^*)$ with $t^*$ as in~\eqref{tstar}. 

		Consider the sets
		\begin{align*}
			C_1&=\{t\in[a,b]\cap N_g^-: t\in (C_g\cap[a,t))'\},\\
			C_2&=\{t\in[a,b]\cap (N_g^+\cup D_g): t\in (C_g\cap(t,b])'\},\\
			C_3&=\{t\in[a,b]\backslash (C_g\cup N_g\cup D_g): t\in (C_g\cap[a,b])'\}.
		\end{align*}	
Then, for $t\in[a,b]$:
	\begin{itemize}
		\item[\textup{1.}] If $t^*\in C_1\cup C_2\cup C_3$, then $f$ is $g$-differentiable at $t$ if and only if $f^*$ is $g$-differentiable at $t$ and
			\begin{equation}\label{condHimpliesFleft}
				\lim_{\substack{s\to t^*\\ s\in C_g}}\frac{f(s)-f(t^*)}{g(s)-g(t^*)}=(f^*)'_g(t),
			\end{equation}
			where we might be considering the corresponding side limit according to the definition of Stieltjes derivative, see \textup{Remark~\ref{remNgderivative}}.
		\item[\textup{2.}] If $t^*\not\in C_1\cup C_2\cup C_3$, $f$ is $g$-differentiable at $t$ if and only if $f^*$ is $g$-differentiable at $t$.
	\end{itemize}
	Furthermore, if $f$ and $f^*$ are $g$-differentiable at $t\in[a,b]$, then $(f^*)'_g(t)=f'_g(t)=f'_g(t^*)$.
\end{pro}
\begin{proof}
		First, observe that given the definition of the Stieltjes derivative at the points of $C_g$, it is enough to prove the result for $t\in[a,b]\backslash C_g$, for which we will use the information in Remark~\ref{pointcharact}.

	Let $t\in[a,b]\backslash C_g$. We define $A_t=\{s\in[a,b]: g(s)\not=g(t)\}$ and $F_t, F^*_t:A_t\to\mathbb F$ as
	\[F_t(s)=\frac{f(s)-f(t)}{g(s)-g(t)},\quad F^*_t(s)=\frac{f^*(s)-f^*(t)}{g(s)-g(t)}=\frac{f^*(s)-f(t)}{g(s)-g(t)},\]
	where the last equality holds since $t^*=t$ for such point. Now, since $t\in[a,b]\backslash C_g$, Remark~\ref{pointcharact} guarantees that~\eqref{cond1} and/or~\eqref{cond2} must hold. 

	Let us assume that~\eqref{cond1} holds. 

	We shall assume that $t\not=a$ as the $g$-derivatives in such point are computed as the right handside limit so it is irrelevant if~\eqref{cond1} holds in that case. In these conditions, we have that $[a,t)\subset A_t$.
	We claim that
	\begin{align}
		\lim_{s\to t^-} F^*_t(s)=f'_g(t),&\quad\mbox{if $f'_g(t)$ exists},\label{FtoHleftlimit}\\
		\lim_{s\to t^-} F_t(s)=(f^*)'_g(t),&\quad\mbox{if $(f^*)'_g(t)$ exists and, if $t\in ([a,t)\cap C_g)'$, }\lim_{\substack{s\to t^-\\ s\in C_g}}F_t(s)=(f^*)'_g(t).\label{HtoFleftlimit}
	\end{align}

	If $t\not\in ([a,t)\cap C_g)'$, there exists $r\in(0,t-a)$ such that $(t-r,t)\cap C_g=\emptyset$. Hence, by definition of $f^*$, we have that $f^*=f$ on $(t-r,t)$ which, in turn, implies that $F^*_t=F_t$ on that set, from which~\eqref{FtoHleftlimit} and~\eqref{HtoFleftlimit} follow. 

	Otherwise, we have that $t\in ([a,t)\cap C_g)'$. In that case, $t\in([a,t)\backslash C_g)'$ as well as, if this was not the case, there would be $\varepsilon\in(0,t-a)$ such that $(t-\varepsilon,t)\backslash C_g=\emptyset$, which would imply that $(t-\varepsilon,t)\subset C_g$ and would contradict~\eqref{cond1}. 
	Now, it is clear that~\eqref{FtoHleftlimit} and~\eqref{HtoFleftlimit} hold if the following statements are true:
	\begin{align*}
		\lim_{\substack{s\to t^-\\ s\not\in C_g}} F^*_t(s)=\lim_{\substack{s\to t^-\\ s\in C_g}} F^*_t(s)=f'_g(t),&\quad\mbox{if $f'_g(t)$ exists,}\\
		\lim_{\substack{s\to t^-\\ s\not\in C_g}} F_t(s)=(f^*)'_g(t),&\quad\mbox{if $(f^*)'_g(t)$ exists and }\lim_{\substack{s\to t^-\\ s\in C_g}}F_t(s)=(f^*)'_g(t),
	\end{align*}
	where we can consider all the limits above because $t\in ([a,t)\cap C_g)'\cap ([a,t)\backslash C_g)'$. Observe that $F^*_t=F_t$ on $A_t\backslash C_g$, so it is clear that $\lim_{s\to t^-} F^*_t|_{A_t\backslash C_g}(s)=\lim_{s\to t^-} F_t|_{A_t\backslash C_g}(s)$. Thus, it suffices to show that 
	\begin{equation}
		\lim_{\substack{s\to t^-\\ s\in C_g}} F^*_t(s)=f'_g(t),\quad\mbox{if $f'_g(t)$ exists.}\label{FimpliesH}
	\end{equation}

	Let $\{s_n\}_{n\in\mathbb N}$ be a sequence in $[a,t)\cap C_g$ converging to $t$. 
	For each $n\in \mathbb N$, denote by $s_n^*$ the corresponding number assigned to $s_n$ by~\eqref{tstar}. Observe that, for each $n\in \mathbb N$, we have that $s_n^*\in(s_n,t)$, which implies that the sequence $\{s_n^*\}_{n\in\mathbb N}$ is contained in $[a,t)$ and converges to $t$.
	Furthermore, we have that, by definition and the left-continuity of $g$,
	\[g(s_n)=g(s_n^*),\quad f^*(s_n)=f(s_n^*),\quad n\in\mathbb N.\]
	Hence, 
	\[\lim_{n\to +\infty} F^*_t(s_n)=\lim_{n\to +\infty}\frac{f^*(s_n)-f^*(t)}{g(s_n)-g(s)}=\lim_{n\to +\infty}\frac{f(s_n^*)-f(t)}{g(s_n^*)-g(s)}=f'_g(t),\]
	where the last equality follows from the fact that $f'_g(t)$ exists. Since $\{s_n\}_{n\in\mathbb N}$ was arbitrarily chosen, we have that~\eqref{FimpliesH} holds. Hence, we have that~\eqref{FtoHleftlimit} and~\eqref{HtoFleftlimit} hold. Observe that, in particular, this proves the result if $t\in C_1$.

	We now assume that~\eqref{cond2} holds. 
	Once again, we can assume that $t\not=b$ as the $g$-derivatives in such point are computed as the left handside limit so it is irrelevant if~\eqref{cond2} holds in that case.
	In these conditions, $(t,b]\subset A_t$ and
	we claim that
	\begin{align}
		\lim_{s\to t^+} F^*_t(s)=f'_g(t),&\quad\mbox{if $f'_g(t)$ exists},\label{FtoHrightlimit}\\
		\lim_{s\to t^+} F_t(s)=(f^*)'_g(t),&\quad\mbox{if $(f^*)'_g(t)$ exists and, if $t\in ((t,b]\cap C_g)'$, }\lim_{\substack{s\to t^+\\ s\in C_g}}F_t(s)=(f^*)'_g(t).
		\label{HtoFrightlimit}
	\end{align}

	If $t\not\in ((t,b]\cap C_g)'$, there exists $r\in(0,b-t)$ such that $(t,t+r)\cap C_g=\emptyset$, which once again implies that $F^*_t=F_t$ on that set, from which~\eqref{FtoHrightlimit} and~\eqref{HtoFrightlimit} follow. Otherwise, $t\in ((t,b]\cap C_g)'$ and, following a similar reasoning as before, we can see that $t\in ((t,b]\backslash C_g)'$ and $\lim_{s\to t^+} F^*_t|_{A_t\backslash C_g}(s)=\lim_{s\to t^+} F_t|_{A_t\backslash C_g}(s)$ so it is enough to show that
	\begin{equation}
		\lim_{\substack{s\to t^+\\ s\in C_g}} F^*_t(s)=f'_g(t),\quad\mbox{if $f'_g(t)$ exists.}\label{FimpliesH+}
	\end{equation}

	Let $\{s_n\}_{n\in\mathbb N}$ be a sequence in $(t,b]\cap C_g$ converging to $t$. 
	For each $n\in \mathbb N$, denote by $s_n^*$ the corresponding number assigned to $s_n$ by~\eqref{tstar}. Note that $s_n^*\in (s_n,b]$, $n\in\mathbb N$. We claim that
	\begin{equation}\label{condconvergencesn*}
		\mbox{for each }\varepsilon>0,\mbox{ there exists }N\in \mathbb N\mbox{ such that } 0<s_n^*-s_n<\varepsilon \mbox{ for all } n\ge N.
	\end{equation}

	Suppose that~\eqref{condconvergencesn*} is not true. In that case, we can find $\varepsilon_0>0$ such that $s_n^*-s_n\ge \varepsilon_0$ for all $n\in\mathbb N$. Since $\{s_n\}_{n\in\mathbb N}\subset (t,b]$ converges to $t$, we can find $p\in\bN$ such that if $n\ge p$, then $s_n< s_p$ and
	\[0< s_n-t<\frac{\e_0}{2}.\] 
	Hence, for any $n\ge p$, we have that
	\[s_n^*-s_p= s_n^*-s_n+s_n-t+t-s_p\ge s_n^*-s_n+t-s_p\ge \e_0-\frac{\e_0}{2}=\frac{\e_0}{2}>0,\]
	so $s_n<s_p<s_n^*$. This implies that, for every $n\ge p$, $s_n$ and $s_p$ are in the same connected component of $C_g$ and, as a consequence, $s_n^*=s_p^*$. Hence, $\bigcup_{n\ge p}(s_n,s_p^*)=(t,s_p^*)\ss C_g$, which contradicts that $t\in ((t,b]\backslash C_g)'$. Thus,~\eqref{condconvergencesn*} must be true.


	We can now show that $\{s_n^*\}_{n\in\mathbb N}$ converges to $t$. Indeed, let $\varepsilon>0$. On the one hand,~\eqref{condconvergencesn*} guarantees that there exists $n_1\in\mathbb N$ such that 
	\[0<s_n^*-s_n<\frac{\varepsilon}{2}\quad\mbox{for all }n\in\mathbb N,\ n\ge n_1.\]
	On the other hand, since $\{s_n\}\subset (t,b]$ converges to $t$, there exists $n_2\in\mathbb N$ such that
	\[0<s_n-t<\frac{\varepsilon}{2}\quad\mbox{for all }n\in\mathbb N,\ n\ge n_2.\]
	Hence, if we take $M=\max\{n_1,n_2\}$, for any $n\in\mathbb N$ such that $n\ge M$, we have that
	\[0<s_n^*-t=s_n^*-s_n+s_n-t\le \frac{\varepsilon}{2}+\frac{\varepsilon}{2}=\varepsilon,\]
	that is, $\{s_n^*\}$ converges to $t$. Furthermore, by definition, 
	\[g(s_n)=g(s_n^*),\quad f^*(s_n)=f(s_n^*),\quad n\in\mathbb N,\]
	from which we have that
	\[\lim_{n\to +\infty} F^*_t(s_n)=\lim_{n\to +\infty}\frac{f^*(s_n)-f^*(t)}{g(s_n)-g(s)}=\lim_{n\to +\infty}\frac{f(s_n^*)-f(t)}{g(s_n^*)-g(s)}=f'_g(t),\]
	where the last equality follows from the fact that $f'_g(t)$ exists. Since $\{s_n\}_{n\in\mathbb N}$ was arbitrarily chosen, we have that~\eqref{FimpliesH+} holds. Hence, we have that~\eqref{FtoHrightlimit} and~\eqref{HtoFrightlimit} hold. Observe that, in particular, this proves the result if $t\in C_2$.

	The remaining cases for $t\in[a,b]\backslash C_g$, namely when~\eqref{cond1} and~\eqref{cond2} hold simultaneously, now follow since, in that case,~\eqref{FtoHleftlimit}-\eqref{HtoFleftlimit} and~\eqref{FtoHrightlimit}-\eqref{HtoFrightlimit} hold.
\end{proof}
\begin{rem}To visualize the sets $C_1$, $C_2$ and $C_3$, we provide an illustration in Figure~\ref{Cs}.
	\end{rem}
	\begin{figure}[h]
	\centering
	\includegraphics[width=.29\textwidth]{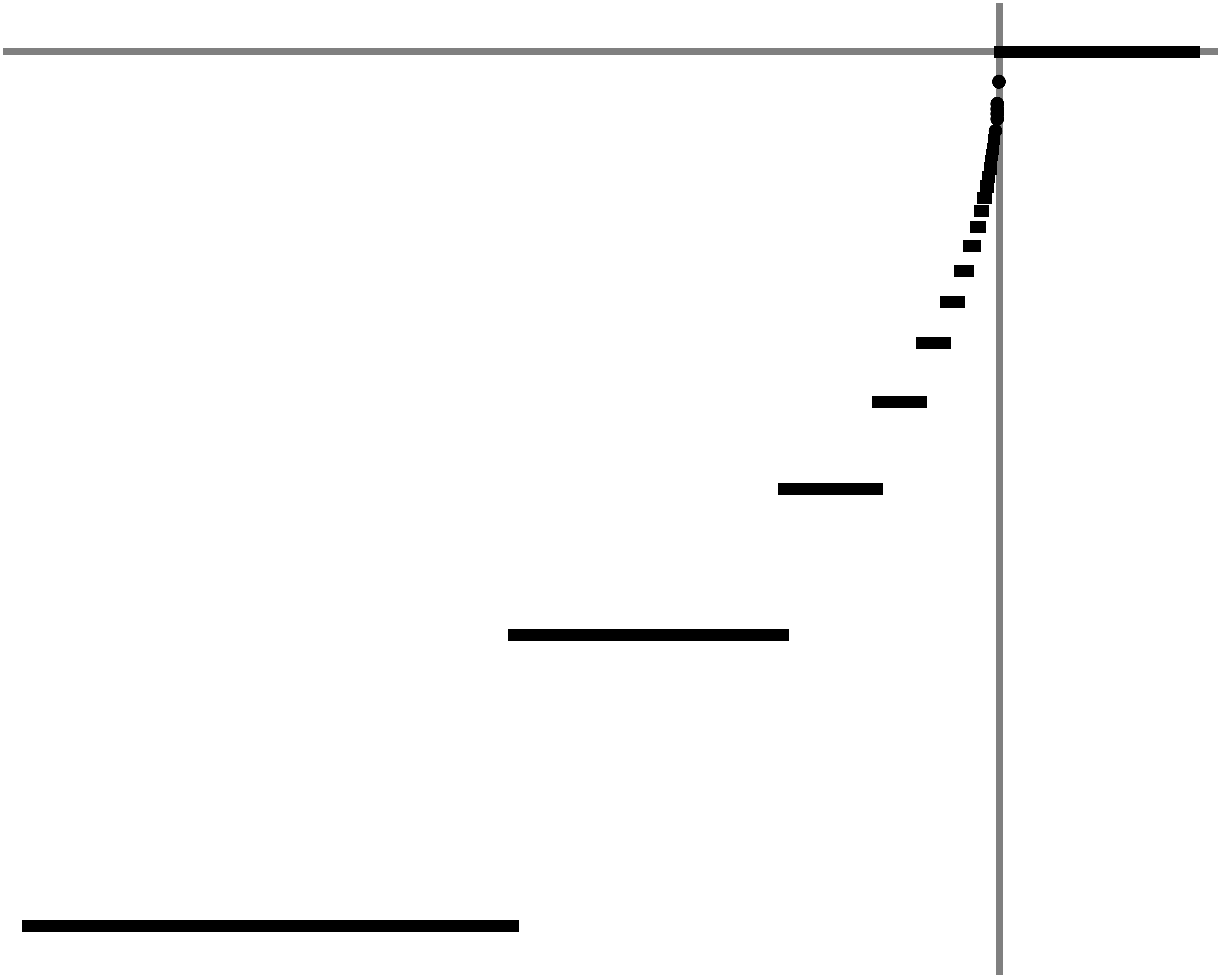}\quad \includegraphics[width=.29\textwidth]{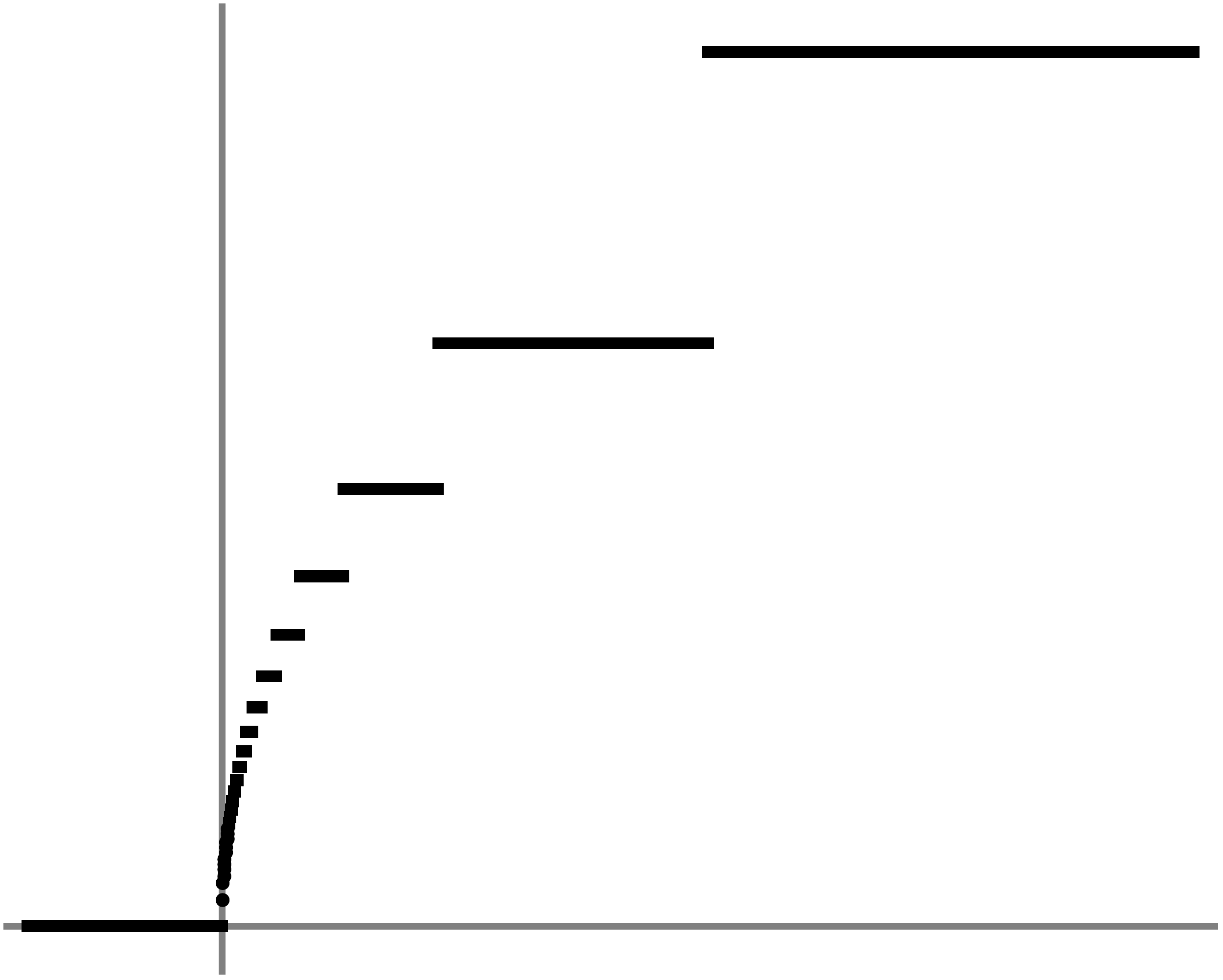}\quad \includegraphics[width=.29\textwidth]{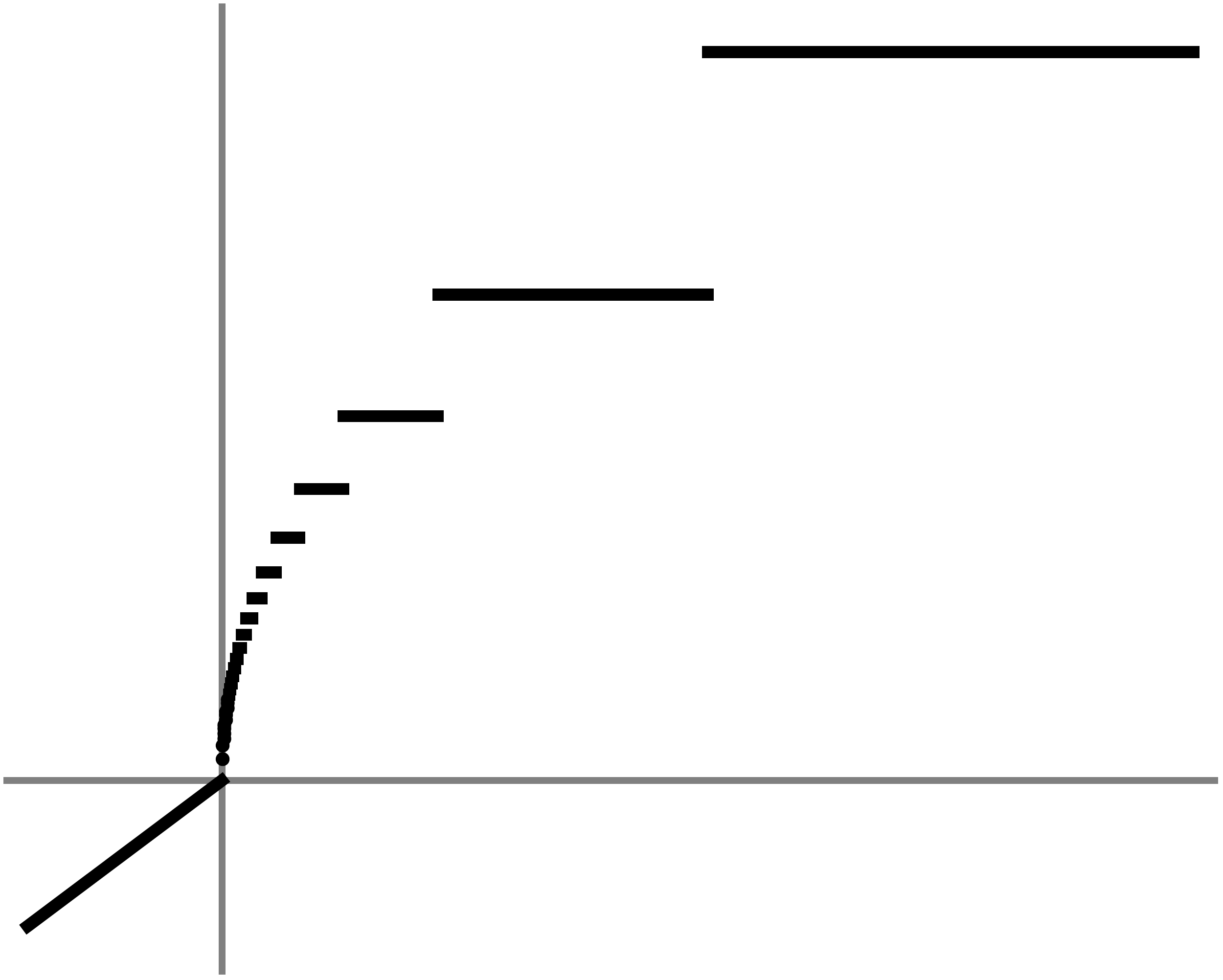}\quad 
	\caption{From left to right, representation of derivators for which zero belongs to $C_1$, $C_2$ and $C_3$ respectively.}
	\label{Cs}
\end{figure}
\begin{rem}\label{remnodif}
	Observe that, in particular, we have that $f^*$ is $g$-differentiable at every point that $f$ is $g$-differentiable, regardless of which type of points it might be. The converse is not necessarily true because, by considering the map $f^*$, we are losing information on the behavior of $f$ along the points of $C_g$. This is reflected in the extra condition that is required for the converse implication, namely, condition~\eqref{condHimpliesFleft}.
	Without it, we cannot ensure the differentiability of $f$.
	Indeed, consider, for example, the map $\widetilde{g}$ in~\eqref{gtilde} and $f:[-1,1]\to\mathbb R$ defined as
	\[
	f(t)=\begin{dcases}
		1,&\displaystyle\mbox{if }t\in\bigcup_{n=1}^\infty \left(\frac{1}{n+1},\frac{1}{n}\right),\vspace{0.1cm}\\
		0,&\mbox{otherwise,}
	\end{dcases}
	\]
	and shown in Figure~\ref{f3}.
	\begin{figure}[h]
		\centering
		\includegraphics[width=\textwidth]{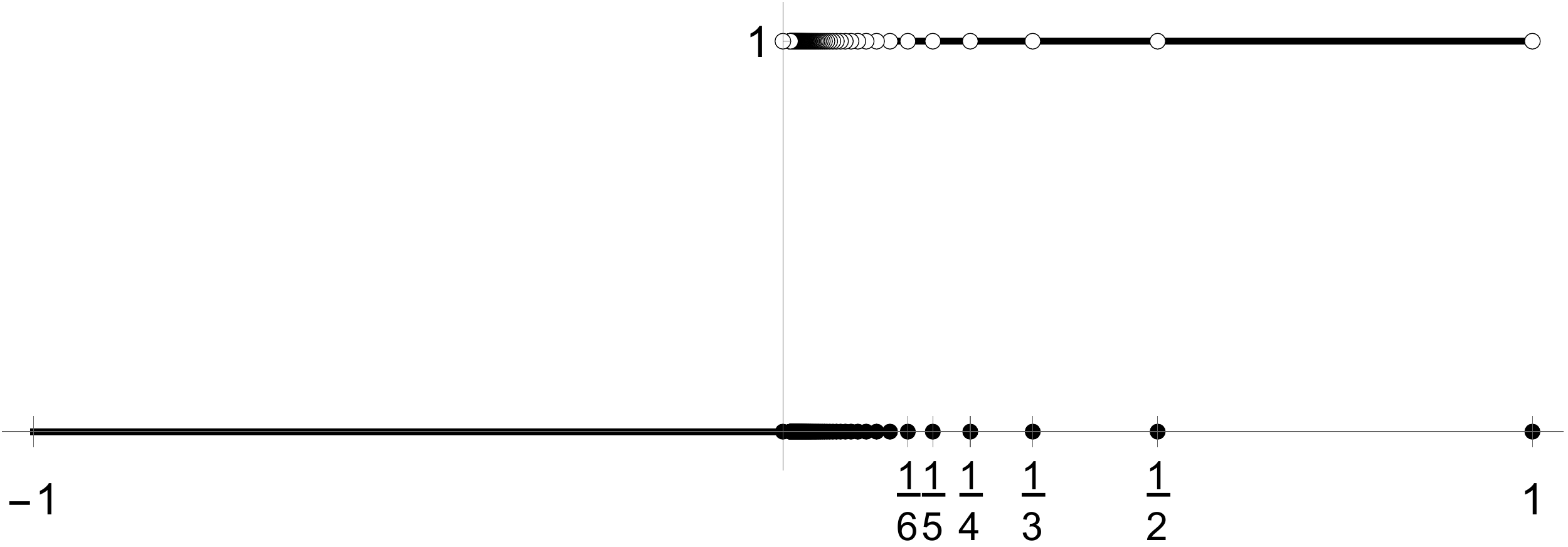}
		\caption{Graph of the function $f$.}
		\label{f3}
	\end{figure}

	Given~\eqref{Cgtilde}, it is easy to see that $f^*(t)=0$, $t\in[-1,1]$, so it follows that $f$ is $g$-differentiable on $[-1,1]$. In particular, it is $g$-differentiable at $0\in N_g^+\cap ((0,1]\cap C_g)'$, which belongs to $C_2$.

	Now, consider the sequence 
	$\{t_n\}_{n\in\mathbb N}=\left\{\frac{1}{2}\(\frac{1}{n}+\frac{1}{n+1}\)\right\}_{n\in\mathbb N}$.
	It is easy to see that it belongs to $(0,1]\cap C_g$, converges to $0$ and, furthermore,
	\[\lim_{n\to+\infty}\frac{f(t_n)-f(0)}{g(t_n)-g(0)}=\lim_{n\to+\infty}\frac{1}{\sum_{k=1}^\infty 2^{-k}\irchi_{(1/k,+\infty)}(t_n)}=\lim_{n\to+\infty}\frac{1}{\sum_{k=n+1}^\infty 2^{-k}}=+\infty.\]
	This implies that $f$ is not $g$-differentiable at $0$ and shows that 
	\eqref{condHimpliesFleft}
	is not satisfied for $t=0$. Similar counterexamples can be constructed for the remaining conditions.

\end{rem}

The following result can be directly deduced from Proposition~\ref{differentiablef*} and Corollary~\ref{differentiableDeltag} and it gives some conditions under which the map $\Delta g(t^*)$ is $g$-dif\-fer\-en\-tiable.

\begin{cor}\label{diffDeltag*}
 Consider the sets $D_1, D_2, D_3$ in~\eqref{defD1}-\eqref{defD3} and define
		$\Delta g^*:[a,b]\to\mathbb R$ as
	\[\Delta g^*(t)=\Delta g|_{[a,b]}(t^*).\] 
	Then, for $t\in[a,b]$:
	\begin{itemize}
		\item[\textup{1.}] If 
$t^*\in D_1\cup D_2\cup D_3$
		and~\eqref{pointwisecondDeltaderivable} holds, then $\Delta g^*$ is $g$-differentiable at $t$.
		\item[\textup{2.}] If 
$t^*\not \in D_1\cup D_2\cup D_3$,
		$\Delta g^*$ is $g$-differentiable at $t$.
	\end{itemize}
	In particular, $\Delta g^*$ is $g$-differentiable on $[a,b]$ if~\eqref{condDeltaderivable} holds.
	Furthermore, if $(\Delta g^*)'_g(t)$ exists, we have that
	\begin{equation*}
		\left(\Delta g^*\right)'_g(t)=-\irchi_{D_g}(t^*).
	\end{equation*}
\end{cor}

Since we have some conditions guaranteeing the Stieltjes differentiability of all the maps involved in~\eqref{deltaproducto}, we can finally obtain a formula for the second Stieltjes derivative of the product of two functions, as presented in the following result.

\begin{pro}\label{twotimesproduct}
 Consider the sets $D_1, D_2, D_3$ in~\eqref{defD1}-\eqref{defD3} and let $t\in[a,b]$ and $f_1,f_2:[a,b]\to\mathbb F$ be two times $g$-differentiable at $t$, then:
	\begin{itemize}
		\item If $t^*\in D_1\cup D_2\cup D_3$
		and~\eqref{pointwisecondDeltaderivable} holds, then $f_{1} f_{2}$ is two times $g$-differentiable at $t$ and
		\begin{align}
			\left(f_{1} f_{2}\right)_{g}''(t)=&\ (f_{1})_{g}''(t) f_{2}(t^*)+f_{1}(t^*)(f_{2})_{g}''(t)+(2-\irchi_{D_g}(t^*))(f_{1})_{g}'(t)(f_{2})_{g}'(t)\nonumber\\
			&+\Delta g(t^*)((f_1)_g'(t)(f_{2})_{g}''(t)+(f_2)_g'(t)(f_{1})_{g}''(t)).\label{secondorderproductrule}
		\end{align}
		\item If $t^*\not \in D_1\cup D_2\cup D_3$
		then $f_{1} f_{2}$ is two times $g$-differentiable at $t$ and~\eqref{secondorderproductrule} holds.
	\end{itemize}
\end{pro}
\begin{proof}
	First, note that Proposition~\ref{PropStiDer} ensures that $f_{1} f_{2}$ is $g$-differentiable at $t$ and
	\begin{align*}
		\left(f_{1} f_{2}\right)_{g}'(t)&=\left(f_{1}\right)_{g}'(t) f_{2}(t^*)+\left(f_{2}\right)_{g}'(t) f_{1}(t^*)+\left(f_{1}\right)_{g}'(t)\left(f_{2}\right)_{g}'(t) \Delta g(t^*)\\
		&=\left(f_{1}\right)_{g}'(t) f_{2}^*(t)+\left(f_{2}\right)_{g}'(t) f_{1}^*(t)+\left(f_{1}\right)_{g}'(t)\left(f_{2}\right)_{g}'(t) \Delta g^*(t).
	\end{align*}
	Furthermore, since $f_1$ and $f_2$ are $g$-differentiable at $t$, Proposition~\ref{differentiablef*} guarantees that $f_1^*$ and $f_2^*$ are also $g$-differentiable at $t$. Now, Corollary~\ref{diffDeltag*} ensures that $\Delta g^*$ is $g$-differentiable under the corresponding conditions. Hence, Proposition~\ref{PropStiDer} guarantees that $(f_{1} f_{2})'_g$ is $g$-differentiable at $t$ and, differentiating the expression in~\eqref{deltaproducto}, and denoting $\irchi^*_{D_g}(t)=\irchi_{D_g}(t^*)$, we have that
	\begin{align*}
		(f_{1} f_{2})_{g}''(t)=&\ (f_{1})_{g}''(t) f_{2}^*(t)+ (f_{1})_{g}'(t)(f_{2}^*)_{g}'(t)+(f_1)_g''(t)(f_{2}^*)_{g}'(t)\Delta g^*(t)\\
		&+(f_{1}^*)_{g}'(t)(f_{2})_{g}'(t)+f_{1}^*(t)(f_{2})_{g}''(t)+(f_1^*)_g'(t)(f_{2})_{g}''(t)\Delta g^*(t)\\
		&+((f_{1})_{g}''(t)(f_{2})_{g}'(t)+(f_{1})_{g}'(t)(f_{2})_{g}''(t)+(f_{1})_{g}''(t)(f_{2})_{g}''(t)\Delta g^*(t))\Delta g^*(t)\\
		&-((f_{1})_{g}''(t)(f_{2})_{g}'(t)+(f_{1})_{g}'(t)(f_{2})_{g}''(t)+(f_{1})_{g}''(t)(f_{2})_{g}''(t)\Delta g^*(t))\irchi^*_{D_g}(t)\Delta g^*(t)\\
		&-(f_{1})_{g}'(t)(f_{2})_{g}'(t)\irchi^*_{D_g}(t).
	\end{align*}
	Now, noting that, $\irchi^*_{D_g}(t)\Delta g^*(t)=\Delta g^*(t)$ and $(f_i^*)'_g(t)=(f_i)'_g(t)$, $i=1,2$, it follows that
	\begin{align*}
		(f_{1} f_{2})_{g}''(t)=&\ (f_{1})_{g}''(t) f_{2}^*(t)+ (f_{1})_{g}'(t)(f_{2})_{g}'(t)+(f_1)_g'(t)(f_{2})_{g}''(t)\Delta g^*(t)\\
		&+(f_{1})_{g}'(t)(f_{2})_{g}'(t)+f_{1}^*(t)(f_{2})_{g}''(t)+(f_1)_g'(t)(f_{2})_{g}''(t)\Delta g^*(t)-(f_{1})_{g}'(t)(f_{2})_{g}'(t)\irchi^*_{D_g}(t)\\
		=&\ (f_{1})_{g}''(t) f_{2}^*(t)+f_{1}^*(t)(f_{2})_{g}''(t)+(2-\irchi^*_{D_g}(t))(f_{1})_{g}'(t)(f_{2})_{g}'(t)\\
		&+\Delta g^*(t)((f_1)_g'(t)(f_{2})_{g}''(t)+(f_1)_g'(t)(f_{2})_{g}''(t)),
	\end{align*}
	which finishes the proof.
\end{proof}

\begin{rem}
	Observe that~\eqref{secondorderproductrule} yields the usual expression of the second derivative of a product of two functions when $D_g=\emptyset$ as, in that case, $\irchi_{D_g}(t^*)=\Delta g(t^*)=0$ for all $t\in[a,b]$. In particular, this means that~\eqref{secondorderproductrule} is, in fact, a generalization of the corresponding expression in the setting of the usual derivative, which corresponds to $g=\Id$.
\end{rem}

This result is enough to shed some light upon the question of higher order derivatives of a product of two functions. Given the expression in~\eqref{secondorderproductrule}, in order to have the product of two three-times differentiable functions be three-times differentiable, we would need the map $\irchi_{D_g}^*$ to be $g$-differentiable. Given Corollary~\ref{diffDeltag*}, this would imply that $\Delta g^*$ would be two-times differentiable in the Stieltjes sense. However, this is not the case, as shown in the first example in Remark~\ref{remarkDeltadiff} (observe that in that case $C_g=\emptyset$ so $\Delta g^*=\Delta g$). This means that, in order to have $g$-differentiability for $\irchi_{D_g}^*$, condition~\eqref{pointwisecondDeltaderivable} is not enough and, thus, further conditions would be required to ensure the existence of higher order derivatives of a product of functions. 
	To that end, we include the following result that from which we will derive some information about the differentiability of $\irchi_{D_g}$ and, as a consequence, of $\irchi_{D_g}^*$.

	\begin{pro}\label{firchidiff}
		Consider the sets
		\begin{align}
			\widetilde D_1&=\{t\in[a,b]\cap N_g^-: t\not\in (D_g\cap[a,t))'\},\label{deftildeD1}\\
			\widetilde D_2&=\{t\in[a,b]\cap (N_g^+\cup D_g): t\not\in (D_g\cap(t,b])'\},\label{deftildeD2}\\
			\widetilde D_3&=\{t\in[a,b]\backslash (C_g\cup N_g\cup D_g): t\not\in (D_g\cap[a,b])'\}\label{deftildeD3},
		\end{align}
and a map $h:[a,b]\to\mathbb F$.
		Given $t\in[a,b]$, the map $f:[a,b]\to\mathbb R$ defined as
		\[f(t)=h(t)\irchi_{D_g}(t),\quad t\in[a,b],\]
		is $g$-differentiable at $t$ if $t^*\in \widetilde D_1\cup \widetilde D_2\cup\widetilde D_3$ and
		\begin{equation}\label{firchiDgdif}
			f'_g(t)=
			\begin{cases}
				0,&\mbox{if }t^*\not\in D_g,\\
				\displaystyle-\frac{h(t^*)}{\Delta g(t^*)},&\mbox{if } t^*\in D_g.
			\end{cases}
		\end{equation}
		Similarly, $f^*:[a,b]\to\mathbb R$ defined as
		\[f^*(t)=h(t^*)\irchi_{D_g}(t^*),\quad t\in[a,b],\]
		is $g$-differentiable at every $t\in[a,b]$ such that $t^*\in \widetilde D_1\cup \widetilde D_2\cup\widetilde D_3$ and, in that case, 
		\[(f^*)'_g(t)=
		\begin{cases}
			0,&\mbox{if }t^*\not\in D_g,\\
			\displaystyle-\frac{h(t^*)}{\Delta g(t^*)},&\mbox{if } t^*\in D_g.
		\end{cases}
		\]
	\end{pro}
	\begin{proof}
		Observe that, thanks to Proposition~\ref{differentiablef*}, it is enough to prove the result for $f$ to obtain the desired property for $f^*$. Hence, we shall focus on showing that $f$ can only be $g$-differentiable at those $t\in[a,b]$ such that $t^*\in \widetilde D_1\cup \widetilde D_2\cup\widetilde D_3$. Furthermore, given the definitions of $t^*$ and the Stieltjes derivative, it is enough to show that, given $t\in[a,b]\backslash C_g$, $f$ is $g$-differentiable at every $t\in \widetilde D_1\cup \widetilde D_2\cup\widetilde D_3$.

		Let $t\in[a,b]\backslash C_g$. We shall study some cases separately.

		First, suppose $t\in N_g^-$. In that case, $t\not=a$ and $t\not\in D_g$, so $f(t)=0$ and the $g$-differentiability of $f$ depends on the existence of
		\[\lim_{s\to t^-}\frac{f(s)}{g(s)-g(t)}.\]
		Hence, if $t\in \widetilde D_1$, $t\not\in (D_g\cap[a,t))'$, so there exists $r>0$ such that $D_g\cap (t-r,t)$, which means that $f=0$ on $(t-r,t)$. Thus, the previous limit is zero,
		so $f$ is $g$-differentiable at $t$ and~\eqref{firchiDgdif} holds.

		Suppose now that $t\in N_g^+$. We have that $t\not=b$ and $f(t)=0$, so $f$ is $g$-differentiable at $t$ if and only if the following limit exists:
		\[\lim_{s\to t^+}\frac{f(s)}{g(s)-g(t)}.\]
		If $t\in \widetilde D_2$, then $t\not\in (D_g\cap(t,b])'$, so we find $r>0$ such that $D_g\cap (t,t+r)$, which guarantees that $f=0$ on $(t,t+r)$ so the previous limit equals zero, i.e., $f$ is $g$-differentiable at $t$ and~\eqref{firchiDgdif} holds.

		Next, suppose $t\in[a,b]\backslash (D_g\cup N_g)$. Then $f(t)=0$ so the $g$-differentiability of $f$ comes from the existence of
		\[\lim_{s\to t}\frac{f(s)}{g(s)-g(t)}.\]
		If $t\in \widetilde D_3$, then $t\not\in (D_g\cap[a,b])'$, so there is $r>0$ such that $D_g\cap (t-r,t+r)\backslash\{t\}$. This means that $f=0$ on $(t,t+r)\backslash\{t\}$, so the previous limit equals zero, that is, $f$ is $g$-differentiable at $t$ and~\eqref{firchiDgdif} holds. 

		Finally, suppose $t\in D_g$. In this case, $t\not=b$ and $f(t)=h(t)$. Observe that, as pointed out in Remark~\ref{remDgCg}, it is enough to check if $f(t^+)$ exists. If $t\in \widetilde D_2$, $t\not\in (D_g\cap(t,b])'$, so we find $r>0$ such that $D_g\cap (t,t+r)$, which guarantees that $f=0$ on $(t,t+r)$ and thus $f(t^+)=0$. It follows now that $f$ is $g$-differentiable at $t$ and
		\[f'_g(t)=\frac{f(t^+)-f(t)}{\Delta g(t)}=-\frac{{h(t)}}{\Delta g(t)},\]
		so~\eqref{firchiDgdif} holds.
	\end{proof}
\begin{rem}To visualize the sets $\widetilde D_1$, $\widetilde D_2$ and $\widetilde D_3$, we provide an illustration in Figure~\ref{TDs}.
\end{rem}
\begin{figure}[h]
	\centering
	\includegraphics[width=.29\textwidth]{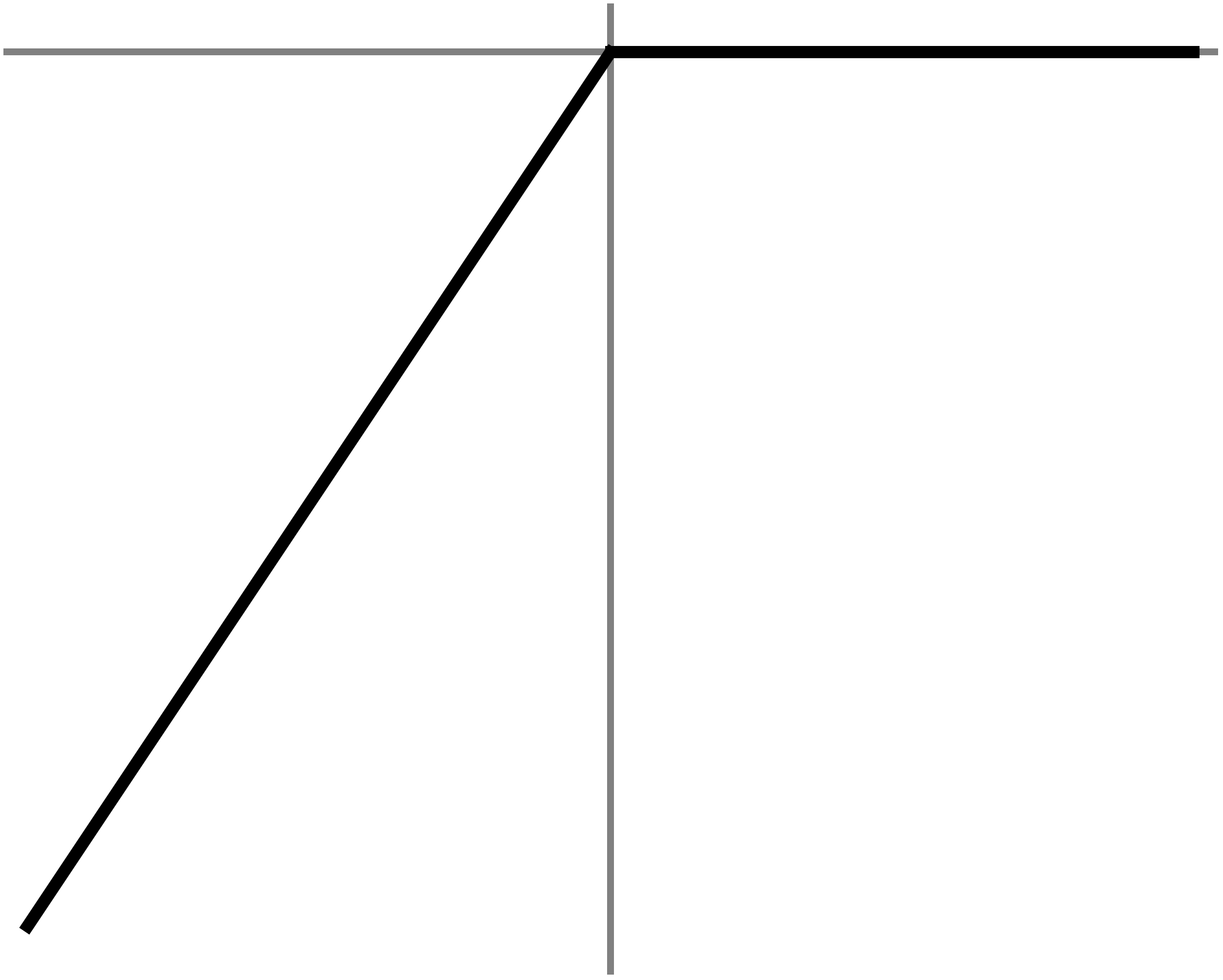}\quad \includegraphics[width=.29\textwidth]{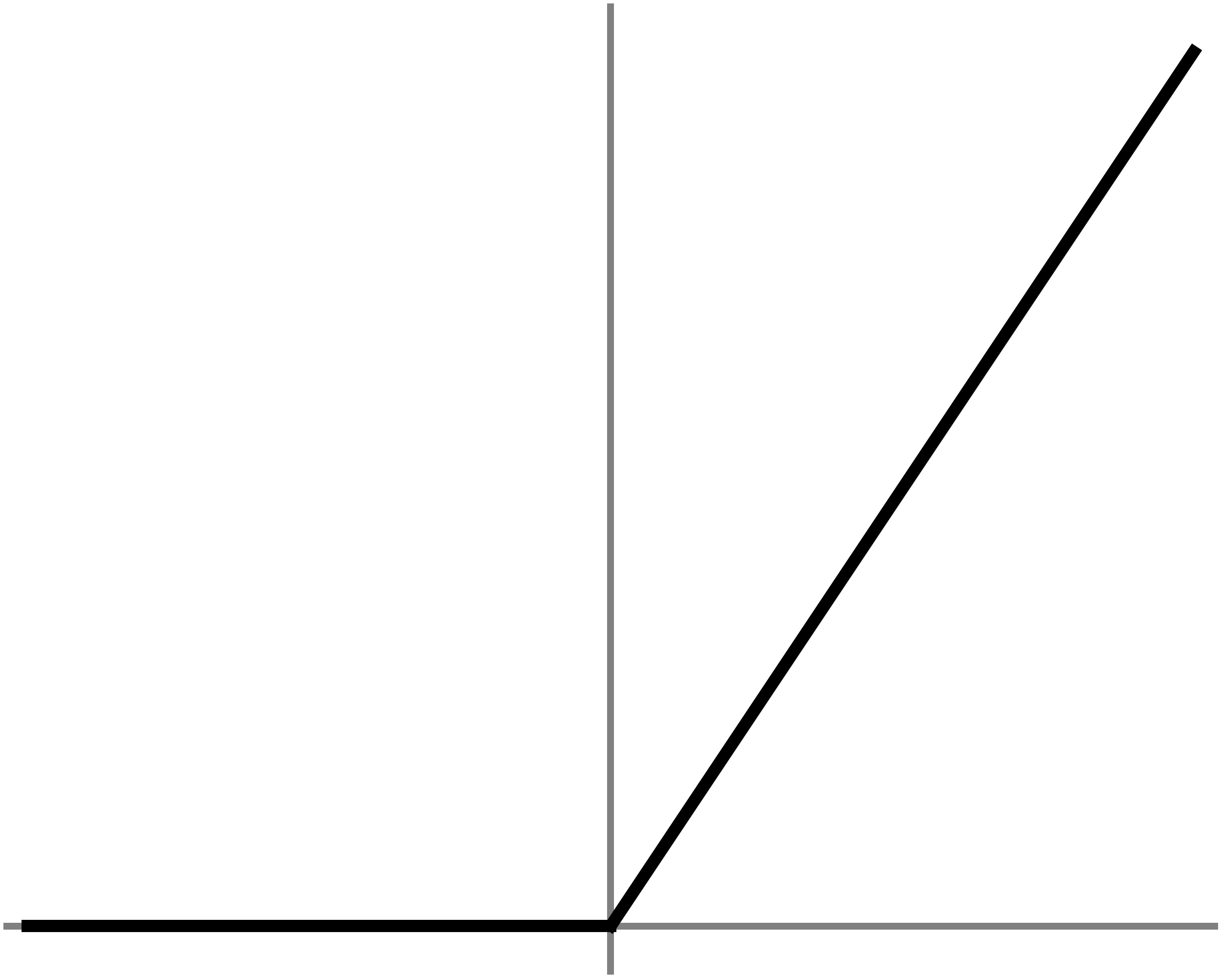}\quad \includegraphics[width=.29\textwidth]{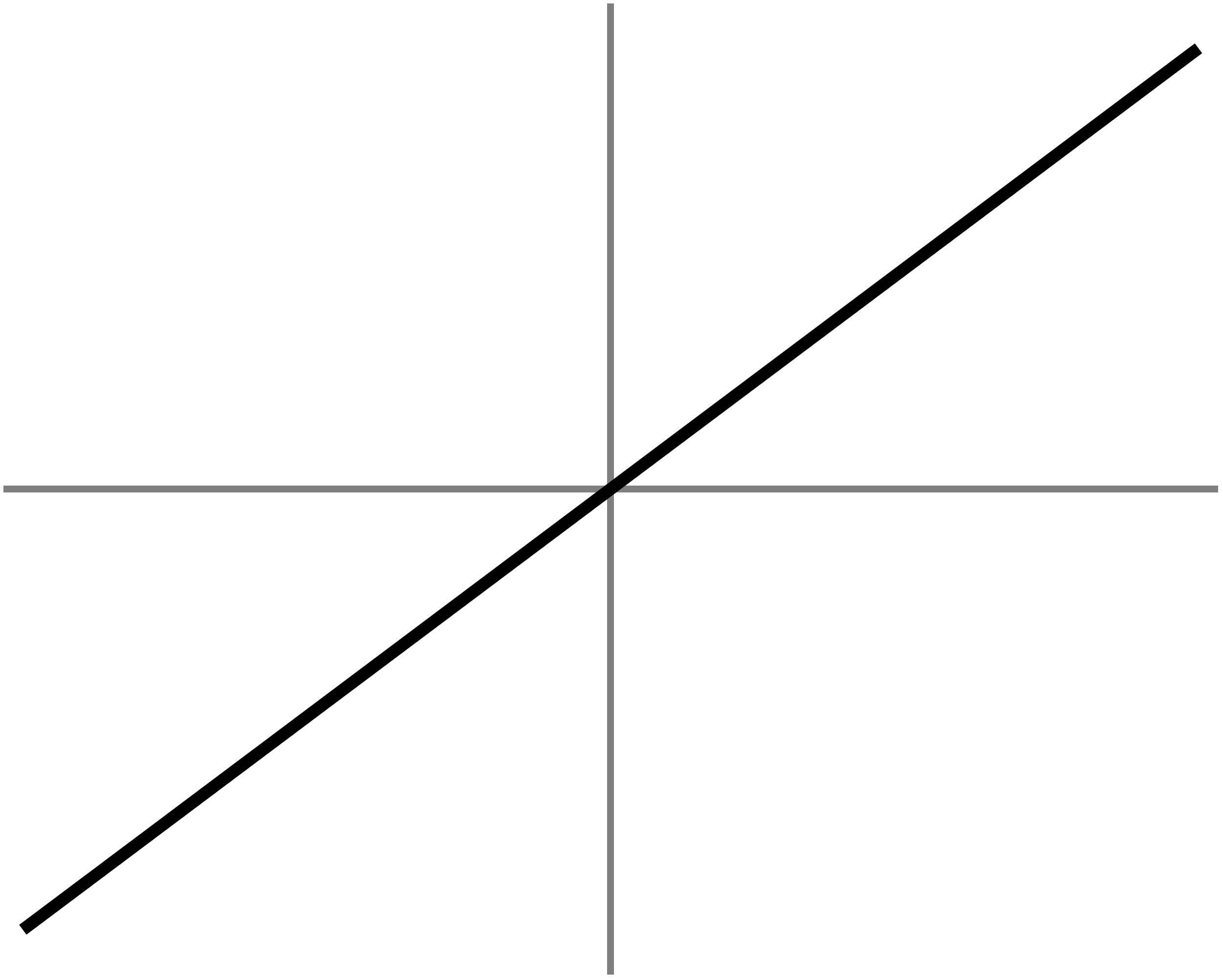}\quad 
	\caption{From left to right, representation of derivators for which zero belongs to $\widetilde D_1$, $\widetilde D_2$ and $\widetilde D_3$ respectively.}
	\label{TDs}
\end{figure}

As a direct consequence, we have the following result characterizating in depth the $g$-dif\-fer\-en\-tiability of $\irchi_{D_g}$ and $\irchi_{D_g}^*$.
\begin{cor}\label{irchidiff}
	Consider the sets $\widetilde D_1,\widetilde D_2, \widetilde D_3$ in~\eqref{deftildeD1}-\eqref{deftildeD3}. The maps $\irchi_{D_g}, \irchi_{D_g}^*$ are $g$-differentiable at every $t\in[a,b]$ such that $t^*\in \widetilde D_1\cup \widetilde D_2\cup\widetilde D_3$ and, for such points,
	\[
		(\irchi_{D_g})'_g(t)=(\irchi_{D_g}^*)'_g(t)=
		\begin{cases}
			0,&\mbox{if }t^*\not\in D_g,\\
			-(\Delta g(t))^{-1},&\mbox{if } t^*\in D_g.
		\end{cases}
\]
	If $t\in[a,b]$ is such that $t^*\not\in \widetilde D_1\cup \widetilde D_2\cup\widetilde D_3$, then $\irchi_{D_g}$ is not $g$-differentiable at $t$.
\end{cor}
\begin{proof}
	The first part of the result follows directly from Proposition~\ref{firchidiff}, so it only remains to show that $\irchi_{D_g}$ fails to be $g$-differentiable for every $t\in[a,b]$ such that $t^*\not\in \widetilde D_1\cup \widetilde D_2\cup\widetilde D_3$. To that end, it is enough to prove that it is not $g$-differentiable at every $t\in[a,b]\backslash C_g$ such that $t\not\in \widetilde D_1\cup \widetilde D_2\cup\widetilde D_3$. 

	First, suppose $t\not\in \widetilde D_1$. In that case, $t\not\in D_g$, so $\irchi_{D_g}(t)=0$. Furthermore, $t\in(D_g\cap [a,t))'$ so $\irchi_{D_g}$ is not $g$-differentiable at $t$ since
	\begin{equation}\label{irchifail1}
			\lim_{\substack{s\to t^-\\ s\in D_g}} \frac{\irchi_{D_g}(s)-\irchi_{D_g}(t)}{g(s)-g(t)}=\lim_{\substack{s\to t^-\\ s\in D_g}} \frac{\irchi_{D_g}(s)}{g(s)-g(t)}=-\infty.
	\end{equation}

	Next, suppose $t\not\in \widetilde D_2$. Similarly, we have that $t\not\in D_g$ and $t\in(D_g\cap (t,b])'$ so $\irchi_{D_g}$ is not $g$-differentiable at $t$ since
	\begin{equation}\label{irchifail2}
		\lim_{\substack{s\to t^+\\ s\in D_g}} \frac{\irchi_{D_g}(s)-\irchi_{D_g}(t)}{g(s)-g(t)}=\lim_{\substack{s\to t^+\\ s\in D_g}} \frac{\irchi_{D_g}(s)}{g(s)-g(t)}=+\infty.
	\end{equation}

Finally, suppose $t\not\in \widetilde D_3$. Then $t\not\in D_g$ and $t\in(D_g\cap [a,b])'$ so $\irchi_{D_g}$ is not $g$-differentiable at $t$ as~\eqref{irchifail1} and/or~\eqref{irchifail2} hold.
\end{proof}

	\begin{rem}
		Observe that we do not talk about the $g$-differentiability of $\irchi^*_{D_g}$ beyond the set $\widetilde D_1\cup\widetilde D_2\cup \widetilde D_3$. This is because if $C_g=\emptyset$, $\irchi^*_{D_g}=\irchi_{D_g}$ which is not $g$-differentiable outside of the mentioned set. Thus, we limit ourselves to the set in which we can guarantee differentiability for all possible nondecreasing and left-continuous functions $g$.
	\end{rem}

	We are now in position to obtain the formula for the third derivative of a product of two three-times $g$-differentiable functions. To that end, 
	we first need to note 
	that the set $D_1, D_2, D_3$ in~\eqref{defD1}-\eqref{defD3}, and the sets $\widetilde D_1,\widetilde D_2, \widetilde D_3$ satisfy that $D_i\cap \widetilde D_i=\emptyset$, $i=1,2,3$. With these relations, we can establish the following result.

	\begin{pro}
		Let $\widetilde D_1, \widetilde D_2, \widetilde D_3$ be as in~\eqref{deftildeD1}-\eqref{deftildeD3} and $f_1,f_2:[a,b]\to\mathbb F$ be three times $g$-differentiable at $t\in[a,b]$ such that $t^*\in \widetilde D_1\cup \widetilde D_2\cup\widetilde D_3$. Then $f_{1} f_{2}$ is three times $g$-differentiable at $t$ and
		\begin{align}
			(f_1f_2)'''_g(t)=&\ (f_1)'''_g(t)f_2(t^*)+f_1(t^*)(f_2)'''_g(t)+(3-\irchi_{D_g}(t^*))((f_1)_g'(t)(f_{2})_{g}''(t)+(f_2)_g'(t)(f_{1})_{g}''(t))\nonumber\\
			&+ \Delta g(t^*)((f_1)'''_g(t)(f_2)'_g(t)+2(f_{1})_{g}''(t)(f_{2})_{g}''(t)+(f_1)'_g(t)(f_2)'''_g(t))\nonumber\\
			&+Q(t^*)(f_{1})_{g}'(t)(f_{2})_{g}'(t)\label{thirdorderproductrule}
		\end{align}
		with
		\begin{equation}\label{quotientDeltag}
			Q(t)=
		\begin{cases}
			0,&\mbox{if }t\not\in D_g,\\
			(\Delta g(t))^{-1},&\mbox{if } t\in D_g.
		\end{cases}
		\end{equation}
	\end{pro}
	\begin{proof}
		Let $t\in[a,b]$ be such that $t^*\in \widetilde D_1\cup \widetilde D_2\cup\widetilde D_3$. Then, $t^*\not\in D_1\cup D_2\cup D_3$ for $D_1, D_2, D_3$ as in~\eqref{defD1}-\eqref{defD3}, so Proposition~\ref{twotimesproduct} ensures that $f_1f_2$ is two times $g$-differentiable at $t$ and $(f_1f_2)''_g(t)= A(t)+B(t)+C(t)$ with
		\begin{align*}
			A(t)&=(f_{1})_{g}''(t) f_{2}^*(t)+f_{1}^*(t)(f_{2})_{g}''(t),&t\in[a,b],\\
			B(t)&=(2-\irchi_{D_g}^*(t))(f_{1})_{g}'(t)(f_{2})_{g}'(t),&t\in[a,b],\\
			C(t)&=\Delta g^*(t)((f_1)_g'(t)(f_{2})_{g}''(t)+(f_2)_g'(t)(f_{1})_{g}''(t)),&t\in[a,b].
		\end{align*}
		Now, Proposition~\ref{differentiablef*} ensures that $f_1^*$ and $f_2^*$ are $g$-differentiable at $t$; whereas Corollary~\ref{diffDeltag*} and Corollary~\ref{irchidiff} guarantee the same property for $\Delta g^*$ and $\irchi^*_{D_g}$, so it follows from Proposition~\ref{PropStiDer} 
		that $f_1f_2$ is three times $g$-differentiable at $t$. Hence, we need to check that~\eqref{thirdorderproductrule} holds.

		First, using~\eqref{deltaproducto} it is easy to see that
		\begin{align*}
			A'_g(t)=&\ (f_1)'''_g(t)f_2(t^*)+f_1(t^*)(f_2)'''_g(t)+(f_1)''_g(t)(f_2)'_g(t)+(f_1)'_g(t)(f_2)''_g(t)\\
			&+\Delta g(t^*)((f_1)'''_g(t)(f_2)'_g(t)+(f_1)'_g(t)(f_2)'''_g(t)).	
		\end{align*}
		Next, it follows from Propositions~\ref{propfDeltadif} and~\ref{differentiablef*} that
		\[C'_g(t)=-\irchi_{D_g}(t^*)((f_1)_g'(t)(f_{2})_{g}''(t)+(f_2)_g'(t)(f_{1})_{g}''(t)).\]
		Finally,
		Corollary~\ref{irchidiff} and~\eqref{deltaproducto} ensure that
		\begin{align*}
			B'_g(t)=&\ Q(t^*)(f_{1})_{g}'(t)(f_{2})_{g}'(t)\\
			&+(2-\irchi_{D_g}(t^*))((f_{1})_{g}''(t)(f_{2})_{g}'(t)+(f_{1})_{g}'(t)(f_{2})_{g}''(t)+(f_{1})_{g}''(t)(f_{2})_{g}''(t)\Delta g(t^*))\\
			&+Q(t^*)((f_{1})_{g}''(t)(f_{2})_{g}'(t)+(f_{1})_{g}'(t)(f_{2})_{g}''(t)+(f_{1})_{g}''(t)(f_{2})_{g}''(t)\Delta g(t^*))\Delta g(t^*)\\
			=&\ Q(t^*)(f_{1})_{g}'(t)(f_{2})_{g}'(t)\\
			&+2 ((f_{1})_{g}''(t)(f_{2})_{g}'(t)+(f_{1})_{g}'(t)(f_{2})_{g}''(t)+(f_{1})_{g}''(t)(f_{2})_{g}''(t)\Delta g(t^*)),		
		\end{align*}
		where the last equality follows from the fact that $Q(t^*)\Delta g(t^*)=\irchi^*_{D_g}(t)$, $t\in[a,b]$. Thus, by the linearity of the Stieltjes derivative, $(f_1f_2)'''_g(t)=A'_g(t)+B'_g(t)+C'_g(t)$ so basic computations yield~\eqref{thirdorderproductrule}.
	\end{proof}

	Given~\eqref{thirdorderproductrule}, it is easy to see that the existence of derivatives of a product of order higher than three depends on the Stieltjes derivative of the map $Q$ in~\eqref{quotientDeltag}.
 Note that Proposition~\ref{firchidiff} shows that $Q$ is $g$-differentiable at $t$ whenever $t^*\in \widetilde D_1\cup\widetilde D_2\cup \widetilde D_3$ and, furthermore, its derivative is another function the same characteristics. Hence, restricting ourselves to such points, it follows that the existence of higher order derivatives of the product of two functions depends exclusively on the order of differentiability of the functions in the product.
 Observe that this happens, in particular, if $D_g'=\emptyset$, which gives an easy condition for the product rule to preserve the order of differentiability.

\section{Regularity of the product}

	In this final section, we shall focus on the issues with the regularity of the product of two functions that arise as a consequence of the product rule,~\eqref{deltaproducto}.	 Specifically, we aim to find the minimal conditions that ensure that the product of two functions in $\mathcal{BC}_g^1([a,b],\mathbb F)$ remains in such set. As a consequence, throughout this section, we shall assume that $a,b\in\mathbb R$, $a<b$, are such that $a\notin N_g^-\cup D_g$ and $b\notin C_g\cup N_g^+\cup D_g$ as this under the conditions under which the set $\mathcal{BC}_g^1([a,b],\mathbb F)$ is defined.

	Given $f_1,f_2\in\mathcal{BC}_g^1([a,b],\mathbb F)$ and~\eqref{deltaproducto}, it is clear that the term preventing their product from being another function in $\mathcal{BC}_g^1([a,b],\mathbb F)$ is $(f_1)'_g(f_2)'_g\Delta g^*$. In particular, it is the $g$-continuity of $\Delta g^*$ that is interfering with the regularity. Thus, we start this section by studying the set of points in which $\Delta g^*$ is $g$-continuous, 

	which we do through the following result that connects the $g$-continuity of a function $f$ with the $g$-continuity of the modified map $f^*$.
	\begin{pro}\label{continuousf*}
		If $f:[a,b]\to\mathbb F$ is $g$-continuous at $t\in [a,b]$, then $f^*$ is also $g$-continuous at $t$.
	\end{pro}
	\begin{proof}
		We prove the result in terms of the sequential formulation of the $g$-continuity,~\eqref{gseqcont}.

		Let $\{t_n\}_{n\in\mathbb N}\subset[a,b]$ be such that $g(t_n)\to g(t)$. Since $g(s)=g(s^*)$ for all $s\in\mathbb R$, it follows that $g(t_n^*)\to g(t^*)$. Now, Lemma~\ref{tc} ensures that $f$ is $g$-continuous at $t^*$, so it follows that $f(t_n^*)\to f(t^*)$, i.e., $f^*(t_n)\to f^*(t)$. Since the sequence was arbitrarily chosen, we have that $f^*$ is $g$-continuous at $t$.
	\end{proof}
\begin{rem}
	The converse of Proposition~\ref{continuousf*} is not necessarily true. Indeed, consider the map
	\[
	g(t)=
	\begin{cases}
		t,&\mbox{if }t\in(-\infty,0]\cup(2,+\infty),\\
		1,&\mbox{if }t\in(0,2]\\
	\end{cases}
	\]
	In that case, denoting $f(t)=\Delta g|_{[-1,3]}(t)$, we have that
	\[
	f(t)=
	\begin{cases}
		1,&\mbox{if }t\in\{0,2\},\\
		0,&\mbox{if }t\in[-1,3]\backslash\{0,2\},
	\end{cases}
	\quad\quad\quad 
	f^*(t)=
	\begin{cases}
		1,&\mbox{if }t\in[0,2]\\
		0,&\mbox{if } t\in[-1,0)\cup(2,3].
	\end{cases}
	\]
	Observe that Proposition~\ref{proddg} ensures that $f$ is not $g$-continuous at $1$. However, $f^*$ is $g$-continuous at $1$ as, given $\varepsilon>0$, taking $\delta\in(0,1)$, we have that if $s\in[-1,3]$ is such that $\left\lvert g(s)-g(1)\right\rvert<\delta$ then, necessarily, $s\in(0,2]$, in which case
	$\left\lvert f^*(s)-f^*(1)\right\rvert=0<\varepsilon$.
\end{rem}

	The following result provides a partial converse to Proposition~\ref{continuousf*}.
	\begin{pro}\label{protci}
		Let $f:[a,b]\to\mathbb F$ and assume $f^*$ is $g$-continuous on $[a,b]$. Then $f$ is $g$-continuous on $[a,b]$ if and only if $f(t)=f(s)$ for every $t,s\in[a,b]$ such that $g(t)=g(s)$.
		\end{pro}
\begin{proof}

	Observe that Lemma~\ref{tc} ensures that if $f$ is $g$-continuous at $t$ and $g(t)=g(s)$ then $f(t)=f(s)$, so if $f$ is $g$-continuous on $[a,b]$, we have that $f(t)=f(s)$ for every $t,s\in[a,b]$ such that $g(t)=g(s)$.

		Conversely, suppose that $f(t)=f(s)$ for every $t,s\in[a,b]$ such that $g(t)=g(s)$. We claim that $f=f^*$. Indeed, let $s\in[a,b]$. If $s\not\in C_g$, we have that $f(s)=f^*(s)$ as $s=s^*$. Otherwise, $s\in C_g$ and we have that $g(s)=g(s^*)$ since $g$ is left-continuous, which implies that $f(s)=f(s^*)=f^*(s)$. Since $f=f^*$ and $f^*$ is $g$-continuous, the result follows.
	\end{proof}
\begin{rem}
	It might be tempting to state Proposition~\ref{protci} in a pointwise fashion, but this does not work as the example in Remark~\ref{remnodif} shows. The map $f$ there satisfies that $f(t)=f(0)=0$ for every $t\in\bR$ such that $\widetilde g(t)=\widetilde g(0)$ (that is, $t\le 0$), but it is not $\widetilde g$-continuous at $0$ as the sequence 
		$\{t_n\}_{n\in\mathbb N}=\left\{\frac{1}{2}\(\frac{1}{n}+\frac{1}{n+1}\)\right\}_{n\in\mathbb N}$ is such that $\widetilde g(t_n)\to0=g(0)$ but $f(t_n)\to 1\not=0=f(0)$. 
	However, $f^*=0$ is clearly $\widetilde g$-continuous at $0$.
	\end{rem}

Note that, unfortunately, Proposition~\ref{protci} does not apply in the context of $\Delta g$ and $\Delta g^*$. However, 	
combining the information in Propositions~\ref{proddg} and~\ref{continuousf*}, we can obtain the following result about the $g$-continuity of $\Delta g^*$.

	\begin{pro}\label{contdg*}
		Consider the set $A_g$ in~\eqref{Ag} and
		\begin{equation*}
			H_g=\{t\in\mathbb R: t\in(s_1,s_2]\ \mbox{ for some }s_1,s_2\in D_g\mbox{ such that }(s_1,s_2)\ss C_g\}.
		\end{equation*}
		Then, the restriction $\Delta g^*|_{[a,b]}$ is $g$-continuous at every point of $((a^*,b]\backslash A_g)\cup H_g$ and $g$-discontinuous at every point of $((a^*,b)\cap A_g)\backslash H_g$. Furthermore, $\Delta g^*|_{[a,b]}$ is $g$-continuous on $[a,a^*]$,
		where we are also considering the case where $[a,a^*]$ is a degenerate interval, i.e., $[a,a^*]=\{a\}$. 
	\end{pro}	
	\begin{proof}
		Given Propositions~\ref{proddg} and~\ref{continuousf*}, it is clear that $\Delta g^*|_{[a,b]}$ is $g$-continuous at every $t\in(a^*,b)\backslash A_g$. 
		Now, given that $b\notin C_g\cup N_g^+\cup D_g$, it follows that $b\not\in A_g$, so Propositions~\ref{proddg} and~\ref{continuousf*} are again enough to guarantee the $g$-continuity at $b$.

		Consider $t\in(a^*,b)\cap H_g$ and let us show that $\Delta g^*|_{[a,b]}$ is $g$-continuous at $t$.
		Let $\varepsilon>0$. 
		Since $t\in (a^*,b)\cap H_g$ and $b\not\in D_g$, there exist $s_1,s_2\in D_g\cap [a^*,b)$ such that $t\in(s_1,s_2]$ and $(s_1,s_2)\ss C_g$. Taking $\d=\min\{\Delta g(s_1), \Delta g(s_2)\}>0$, we have that, if $s\in\mathbb R$ is such that $\left\lvert g(s)-g(t)\right\rvert<\d$, then $s\in(s_1,s_2]$, so 
		\[\left\lvert \Delta g^*(s)-\Delta g^*(t)\right\rvert=\left\lvert \Delta g(s^*)-\Delta g(t^*)\right\rvert=\left\lvert \Delta g(s_2)-\Delta g(s_2)\right\rvert=0<\varepsilon,\]
		i.e. $\Delta g^*|_{[a,b]}$ is $g$-continuous at $t$.

		Next, we show that $\Delta g^*|_{[a,b]}$ is $g$-discontinuous at every point of $((a^*,b)\cap A_g)\backslash H_g$, distinguishing between the points that belong to $D_g$, from those which do not.

		First, let $t\in ((a^*,b)\cap A_g)\backslash H_g$ be such that $t\in D_g$. In this case, $\Delta g^*(t)=\Delta g(t)$. Proceeding in a similar fashion to the proof of Proposition~\ref{proddg}, we can have that either~\eqref{cond1} or~\eqref{cond3} holds.

		If~\eqref{cond1} holds, then, necessarily, $s\le s^*< t=t^*$ for every $s<t$. This implies that, given a sequence $\{t_n\}_\n$ converging to $t$ such that $t_n<t$, the sequence $\{t_n^*\}_\n$ has the same properties, so Proposition~\ref{regulated} guarantees that
		\[\lim_{n\to \infty}\Delta g^*(t_n)=\lim_{n\to \infty}\Delta g(t_n^*)=0.\] 
		Since the sequence $\{t_n\}_\n$ was arbitrarily chosen, we have that $\lim_{s\to t^-}\Delta g^*(s)=0$, so we can find $\rho\in(0,t-a)$ such that 
		\[\Delta g^*(s)=\Delta g(s^*)<\frac{\Delta g(t)}{2},\quad s\in\bR,\ 0< t-s<\rho.\]
		Since~\eqref{cond1} holds, so we can find $t_1<t$ such that $t-t_1^*<\rho$ and $g(t_1^*)<g(t)$. Let $\d=g(t)-g(t_1^*)$. Observe that, if $0< g(t)-g(s)<\d$, since $g$ is nondecreasing, then $0<t-s<t-t_1^*<\rho$, and so $\Delta g^*(s)<\Delta g(t)/2$. But this implies that, if $0< g(t)-g(s)<\d$, then
		\[\left\lvert \Delta g^*(s)-\Delta g^*(t)\right\rvert\ge \Delta g^*(t)-\Delta g^*(s)=\Delta g(t)-\Delta g^*(s)>\Delta g(t)-\frac{\Delta g(t)}{2}=\frac{\Delta g(t)}{2},\]
		and, therefore, $\Delta g^*$ is not continuous at $t$.

		On the other hand, if~\eqref{cond3} holds, let us consider
		\[\widetilde t=\inf\{s\in [a^*,b]\ :\ g(s)=g(t)\}.\]
		Observe that $\til t\not\in D_g$, since $t\not\in H_g$, and $\til t>a^*$, given that $a^*\in N_g^+\cup D_g$ by definition. Furthermore, we also have that $g(s)\le g(\til t)\le g(t)$ for every $s<t$ and $g(s)=g(t)$ for every $s\in [\til t,t]$. Now, reasoning as in the previous case, we find $\d\in(0,t-a)$ such that, if $0< g(\til t)-g(s)<\d$, then $\Delta g^*(s)<\Delta g(t)/2$, which, again, is enough to show that $\Delta g^*|_{[a,b]}$ is not $g$-continuous at $t$.

		Consider now $t\in ((a^*,b)\cap A_g)\backslash H_g$ such that $t\not\in D_g$ and consider $\widehat t$ as in~\eqref{hatt}.
		Since $t\in A_g$, we have that $\widehat t\in D_g$ so, given that $t,b\not\in D_g$, it follows that $t<\widehat t<b$. In particular, this implies that $\widehat t\in ((a^*,b)\cap A_g)\backslash H_g$ and $\widehat t\in D_g$, so we already know $\Delta g^*|_{[a,b]}$ is not $g$-continuous at $\til t$. Observe that, since $g$ is left-continuous, we have that $g(t)=g(\widehat t)$, so Lemma~\ref{tc} is enough to ensure that $\Delta g^*|_{[a,b]}$ is not $g$-continuous at $t$, which finishes the proof of the first part of the result.

		Let us focus now on the behaviour of $\Delta g^*|_{[a,b]}$ on $[a,a^*]$. Observe that it is enough to study the $g$-continuity at $a^*$. Indeed, this is trivial when $a^*=a$. When $a^*>a$, it is enough to note that, in that case, $a\in C_g$, which means that 
		\[g(t)=g(a^*),\quad \Delta g^*(t)=\Delta g^*(a)=\Delta g^*(a^*),\quad \quad t\in[a,a^*],\] 
		so the $g$-continuity at $t\in[a,a^*]$ can always be deduced from the $g$-continuity at $a^*$ by Lemma~\ref{tc}.

		First, suppose that $a^*\in D_g$. Given $\varepsilon>0$, take $\delta\in(0,\Delta g(a^*))$. In that case, if $t\in[a,b]$ is such that $\left\lvert g(t)-g(a^*)\right\rvert<\delta$ we necessarily have that $t\in[a,a^*]$, so
		\[\left\lvert \Delta g^*(t)-\Delta g^*(a^*)\right\rvert=0<\varepsilon.\]
		This means that $\Delta g^*|_{[a,b]}$ is $g$-continuous at $a^*$ and, thus, $g$-continuous on $[a,a^*]$.

		Finally, if $a^*\not\in D_g$, then $a^*\not\in A_g$ so $\Delta g$ is $g$-continuous at $a^*$. Thus, by Proposition~\ref{continuousf*}, $\Delta g^*|_{[a,b]}$ is $g$-continuous at $a^*$ which, again, is enough to obtain that it is $g$-continuous on $[a,a^*]$.
	\end{proof}
\begin{rem}To visualize the kind of points in $H_g$, observe that in Example~\ref{exaa}, $A_g=\{-1\}\cup[0,2]$, $H_g=(0,2]$ and $A_g\bs H_g=\{-1,0\}$.
	\end{rem}
%
%
%

We are now finally in position to determine under which conditions the product of two functions in $\mathcal{BC}_g^1([a,b],\mathbb F)$ remains in this set. This is the information in the next result.	

\begin{thm}\label{thmrp}
	Let $f_1,f_2\in\cB\cC^1_g([a,b],\bF)$. Then, $f_1f_2\in\cB\cC^1_g([a,b],\bF)$ if and only if \begin{equation}\label{regularitycond}
		(f_1)'_g(t)(f_2)'_g(t)=0,\quad \mbox{for all }t\in((a^*,b)\cap A_g)\backslash H_g.
	\end{equation}
\end{thm}
\begin{proof}
	First, observe that, since $f_1,f_2\in\cB\cC^1_g([a,b],\bF)$, Proposition~\ref{deltaproducto} ensures that 
	\[\left(f_{1} f_{2}\right)_{g}'(t)=\left(f_{1}\right)_{g}'(t) f_{2}(t^*)+\left(f_{2}\right)_{g}'(t) f_{1}(t^*)+\left(f_{1}\right)_{g}'(t)\left(f_{2}\right)_{g}'(t) \Delta g(t^*),\quad t\in[a,b],\]
	so, given the definition of $t^*$ and that $f_1$, $f_2$ are continuous from the left at every point of $(a,b]$, see Proposition~\ref{proreg}, we have that
	\begin{equation}\label{gcontproductexpr}
		\left(f_{1} f_{2}\right)_{g}'(t)=\left(f_{1}\right)_{g}'(t) f_{2}(t)+\left(f_{2}\right)_{g}'(t) f_{1}(t)+\left(f_{1}\right)_{g}'(t)\left(f_{2}\right)_{g}'(t) \Delta g(t^*),\quad t\in[a,b].
	\end{equation}
	Hence, since all the functions involved in this expression of $\left(f_{1} f_{2}\right)_{g}'$ are bounded, it is clear that we only need to concern ourselves with its $g$-continuity. Furthermore,
	Proposition~\ref{contdg*} guarantees that $\Delta g^*|_{[a,b]}$ is $g$-continuous on $[a,b]\setminus(((a^*,b)\cap A_g)\backslash H_g)$, so it follows that $\left(f_{1} f_{2}\right)_{g}'$ is $g$-continuous on that set. As a consequence, it is enough to show that~\eqref{regularitycond} holds if and only if 
	\begin{equation}\label{regularityaux}
		\left(f_{1} f_{2}\right)_{g}'\mbox{ is $g$-continuous at every $t\in ((a^*,b)\cap A_g)\backslash H_g$}.
	\end{equation}

	First, suppose~\eqref{regularityaux} holds.
	 Reasoning by contradiction, suppose~\eqref{regularitycond} did not hold. In that case, we would find $t\in ((a^*,b)\cap A_g)\backslash H_g$ such $(f_1)'_g(t)(f_2)'_g(t)\not=0$. Since $(f_1)_g'$ and $(f_2)_g'$ are $g$-continuous by hypothesis, we can find $\d>0$ such that
	 \[(f_1)_g'(s)(f_2)_g' (s)\ne 0,\quad s\in I:=\{r\in[a,b],\ \left\lvert g(r)-g(t)\right\rvert<\d\}.\] 
	 Thus, given~\eqref{gcontproductexpr}, we have that
	 \[\Delta g^*(s)=\frac{\left(f_{1} f_{2}\right)_{g}'(s)-\left(f_{1}\right)_{g}'(s) f_{2}(s)-\left(f_{2}\right)_{g}'(s) f_{1}(s)}{(f_1)_g'(s)(f_2)_g' (s)},\quad s\in I.\]
	 Everything on the right hand side of the equation is $g$-continuous at $t$, so, thanks to Lemma~\ref{lemdiv}, we conclude that $\Delta g^*|_{[a,b]}$ is $g$-continuous at $t$, which contradicts Proposition~\ref{contdg*}.

	Conversely, suppose~\eqref{regularitycond} holds. Consider $t\in ((a^*,b)\cap A_g)\backslash H_g$ and let us show that $(f_{1} f_{2})'_{g}$ is $g$-continuous at $t$. Given~\eqref{gcontproductexpr}, it is clear that it is enough to prove that $h:=\left(f_{1}\right)_{g}'\left(f_{2}\right)_{g}' \Delta g^*$ is $g$-continuous at $t$.

By hypothesis, $h(t)=0$. Assume, by contradiction, that $h$ is not $g$-continuous at $t$. Then there exists $\e_0>0$ and a sequence $\{t_n\}_\n\subset [a,b]$ such that $g(t_n)\to g(t)$ and $\left\lvert h(t_n)\right\rvert\ge \e_0$ for every $\n$. 
	Observe that this fact, together with~\eqref{regularitycond},
imply that $\{t_n\}_\n\ss[a,b]\backslash (((a^*,b)\cap A_g)\backslash H_g)$. Furthermore, 
	since $t^*\not\in D_g$ for any $t\not\in A_g$,
we have that $\Delta g^*|_{[a,b]}=0$ on $[a,b]\bs A_g$, so 
\[\{t_n\}_\n \ss A_g\backslash (((a^*,b)\cap A_g)\backslash H_g)=[a,a^*]\cup H_g.\] 
Now, since $t\in ((a^*,b)\cap A_g)\backslash H_g$, 
for any $\d\in (0, g(t)-g(a^*))$ we have that
\[\left\lvert g(t)-g(s)\right\rvert=g(t)-g(s)\ge g(t)-g(a^*)>\delta, \quad s\in[a,a^*].\]
 Therefore, since $g(t_n)\to g(t)$, we conclude that $t_n\in H_g$ for $n$ sufficiently big, so we can assume with loss of generality that $\{t_n\}_\n\ss H_g$. Furthermore, since $h(t_n)=h(t_n^*)$ and $g(t_n)=g(t_n^*)$ for all $n\in\mathbb N$, we can also assume that $t_n=t_n^*$ and, thus, $\{t_n\}_\n\ss D_g$. Now, we necessarily have there are infinitely many distinct terms in $\{t_n\}_\n$ for, otherwise, the sequence would be eventually constant which would imply that $t\in H_g$, which we know is not the case. Finally, since $f_1,f_2\in\cB\cC^1_g([a,b],\bF)$, there exists $M>0$ such that $\left\lvert (f_1)_g'(t_n)\right\rvert,\left\lvert (f_2)_g'(t_n)\right\rvert\le M$. Therefore,
\[\e_0\le \left\lvert h(t_n)\right\rvert=\left\lvert (f_1)_g'(t_n)(f_2)_g'(t_n)\Delta g(t_n)\right\rvert\le M^2\Delta g(t_n),\]
that is, $\Delta g(t_n)\ge \e/M^2$ for every $\n$. This means that there are infinitely many points $s\in [a,b]\cap D_g$ such that $\Delta g(s)\ge \e/M^2$, which is impossible since
\[0\le\sum_{s\in [a,b]\cap D_g} \Delta g(s)= \int_{[a,b]\cap D_g} \dif g(s)\le \int_{[a,b]}\dif g(s)=\mu_g([a,b])=g(b)-g(a)<\infty.\qedhere\]
\end{proof}
\begin{rem}
	Observe that Theorem~\ref{thmrp} generalizes, in some sense, \cite[Proposition~3.17]{Fernandez2021}, where a sufficient condition for the continuity of the derivative of the product was used: that one of the functions involved was continuous in the usual sense. Observe that if $f_1$ is continuous this automatically implies that $(f_1)'_g(t)=0$ for all $t\in[a,b]\cap D_g$, in which case~\eqref{gcontproductexpr} becomes the usual product rule formula. Theorem~\ref{thmrp}, on the other hand, provides a necessary and sufficient condition for the product of two $\mathcal{BC}_g^1([a,b],\mathbb F)$	functions to remain in that set without requiring a simpler version of the product rule.
\end{rem}
\begin{rem}
	Observe that, under the conditions of Theorem~\ref{thmrp}, for $t\in ([a,a^*]\bs A_g)\cup[(a^*,b)\bs H_g]$ we have that
		\[\left(f_{1} f_{2}\right)_{g}'(t)=\left(f_{1}\right)_{g}'(t) f_{2}(t)+\left(f_{2}\right)_{g}'(t) f_{1}(t).\]
		Furthermore, if $f_1\cdot f_2$ is two times differentiable at $t$ (that is, under the hypotheses of Proposition~\ref{twotimesproduct}), then, by~\eqref{secondorderproductrule},
	\[	\left(f_{1} f_{2}\right)_{g}''(t)=\ (f_{1})_{g}''(t) f_{2}(t^*)+f_{1}(t^*)(f_{2})_{g}''(t)+\Delta g(t^*)((f_1)_g'(t)(f_{2})_{g}''(t)+(f_2)_g'(t)(f_{1})_{g}''(t))\]
		for $t\in((a^*,b)\cap A_g)\backslash H_g$. This can serve as a basis to prove a characterization of a product of two class two functions being in the same class similar to Theorem~\ref{thmrp}.
\end{rem}

\section*{Acknowledgments}
The authors were partially supported by Xunta de Galicia, project ED431C 2019/02, and by the Agencia Estatal de Investigaci\'on (AEI) of Spain under grant MTM2016-75140-P, co-financed by the European Community fund FEDER. Ignacio M\'arquez Alb\'es was partially supported by Xunta de Galicia under grant ED481B-2021-074.

 \bibliography{refs-ho}
 \bibliographystyle{spmpsciper}

 \end{document}